\documentclass[11pt,a4paper,reqno]{amsart}

\usepackage{amscd}
\usepackage{amssymb}
\usepackage{amsthm}
\usepackage{graphicx}
\usepackage{color}
\usepackage[all]{xy}
\usepackage{mathrsfs}
\usepackage{marvosym}
\usepackage{scalerel,amssymb}
\usepackage{hyperref}

\newtheorem{theorem}{Theorem}[section]
\newtheorem{lemma}[theorem]{Lemma}

\newtheorem{conjecture}[theorem]{Conjecture}

\newtheorem{example}[theorem]{Example}
\newtheorem{question}[theorem]{Question}

\theoremstyle{remark}
\newtheorem{remark}[theorem]{Remark}
\numberwithin{equation}{section}
\numberwithin{figure}{section}

\newcommand{\ZZ} {\mathbb{Z}}

\newcommand{\RR} {\mathbb{R}}
\newcommand{\CC} {\mathbb{C}}
\newcommand{\PP} {\mathbb{P}}
\newcommand{\VV} {\mathbb{V}}

\newcommand {\cN}  {\mathcal{N}}
\newcommand {\fop}  {\mathfrak{p}}
\newcommand{\cP}{\mathcal{P}}

\newcommand {\Hom}  {\operatorname{Hom}}
\newcommand {\Spec} {\operatorname{Spec}}
\newcommand{\Rea} {\mathrm{Re}}
\newcommand{\Ima} {\mathrm{Im}}
\newcommand{\hooklongrightarrow}{\lhook\joinrel\longrightarrow}

\begin{document}

\title[Holomorphic Floer theory and Donaldson-Thomas invariants]{Holomorphic Floer theory and Donaldson--Thomas invariants}

\author[P.\,Bousseau]{Pierrick Bousseau}
\address{University of Georgia, Department of Mathematics, Athens, GA 30605}
\email{Pierrick.Bousseau@uga.edu}

\date{}

\begin{abstract}
We present several expected properties of the holomorphic Floer theory of a holomorphic symplectic manifold. 
In particular, we propose a conjecture relating holomorphic Floer theory of Hitchin integrable systems and Donaldson-Thomas invariants of non-compact Calabi-Yau 3-folds. More generally, we conjecture that the BPS spectrum of a 4-dimensional
$\mathcal{N}=2$ 
quantum field theory can be recovered from the holomorphic Floer theory of the corresponding Seiberg-Witten integrable system.
\end{abstract}

\maketitle

\setcounter{tocdepth}{1}
\tableofcontents
\setcounter{section}{-1}
\section{Introduction}
The goal of this paper is to present an overview of the expected formal aspects of holomorphic Floer theory of 
holomorphic symplectic manifolds and to propose a conjectural relation between holomorphic Floer theory and Donaldson-Thomas (DT) theory.

\subsection{Holomorphic Floer theory}
\label{sec:intro_HFT}
Holomorphic Floer theory is an example of infinite-dimensional complex Morse theory. We first review the more familiar case of Floer theory as an example of infinite-dimensional Morse theory.
\subsubsection{Floer theory as Morse theory}
Floer theory of a symplectic manifold 
$(M,\omega)$ is formally the result of applying Morse theory to the 
infinite-dimensional manifold $\mathcal{P}$ of paths between pairs of Lagrangian submanifolds $L_1,L_2 \subset M$, endowed with the action functional 
\begin{align*} f: \mathcal{P} &\longrightarrow \RR \\
\mathfrak{p}&\longmapsto -\int_{D_\mathfrak{p}} \omega\,,
\end{align*}
where $D_\mathfrak{p}$ is the surface in $M$ formed by a one-parameter family of paths from a reference path $\mathfrak{p}_0$ to the path $\mathfrak{p}$ \cite{Flo1}.
Critical points of $f$ are constant paths located at the intersection points $L_1 \cap L_2$.
Moreover, the choice of a $\omega$-compatible almost complex structure $J$ on $M$ defines a Riemannian metric on $\mathcal{P}$, and the corresponding
gradient flow lines of $f$ are $J$-holomorphic disks in $M$ with boundary on $L_1 \cup L_2$.
Hence, the Morse complex of $f$ is the Floer complex $CF(L_1,L_2)$ generated by the intersection points $L_1 \cap L_2$ and whose differential is given by counts of $J$-holomorphic disks. 
More generally, counting $J$-holomorphic disks with boundary on finite unions of Lagrangian submanifolds leads to the definition of the Fukaya category of 
$(M,\omega)$, whose objects are Lagrangian submanifolds and spaces of morphisms are Floer homology groups \cite{Fukcat}.

\subsubsection{Complex Morse theory}
Whereas Morse theory deals with a real-valued smooth function on a Riemannian manifold, complex Morse theory considers a holomorphic function on a Kähler manifold. 
This includes the Picard-Lefschetz theory of critical points and vanishing cycles of holomorphic functions, and also its categorification given by the Fukaya-Seidel category \cite{Seid1, Seid2}.

The original definition of the Fukaya-Seidel category is based on the geometry of Lefschetz thimbles, but we will use a conjecturally equivalent alternative definition due to  Haydys \cite{Hayd} and Gaiotto-Moore-Witten \cite{GMWinfrared}\footnote{Haydys's proposal for a definition of the Fukaya-Seidel category was made rigorous very recently by Donghao Wang \cite{FS_rig}. }. 
Given a Kähler manifold $\mathcal{P}$ and a holomorphic Morse function
\[ W \colon 
\mathcal{P} \longrightarrow \CC\,,\]
this definition is based on counts of gradient flow lines of the real Morse functions $\mathrm{Re}(e^{-i\theta} W)$ connecting pairs of critical points, and on counts of solutions to a $W$-perturbed holomorphic curve equation asymptotically approaching gradient flow lines.
In physics, a real-valued Morse function on a Riemannian manifold defines a $\mathcal{N}=2$
supersymmetric quantum mechanics, whose space of vacua is Morse homology \cite{WMorse}, whereas a holomorphic Morse function on a Kähler manifold defines a Landau-Ginzburg
2-dimensional $\mathcal{N}=(2,2)$ supersymmetric quantum field theory, whose category of boundary conditions is the Fukaya-Seidel category \cite{GMWinfrared}.

\subsubsection{Holomorphic Floer theory as complex Morse theory}
The main advantage of the construction of the Fukaya-Seidel category in terms of counts of gradient flow lines and
$W$-perturbed holomorphic curves is that it can be formally applied to infinite dimensional set-ups \cite{Hayd}. Known examples reviewed in more details in 
\S \ref{sec_inf_ex}
include the complex Chern-Simons theory for complex connections on a 3-manifold 
\cite{Hayd, Wanalytic}\cite[\S 4.2]{khan2021categorical} and the holomorphic Chern-Simons theory for $\overline{\partial}$-connections on Calabi-Yau 3-folds \cite{DTgauge, DS, Hayd}.

Holomorphic Floer theory is another example, obtained by taking for $\mathcal{P}$
the space of paths between two holomorphic Lagrangian submanifolds $L_1$ and $L_2$ in a holomorphic symplectic manifold $(M, \Omega)$ and for $W$ the holomorphic action functional 
\begin{align*}
W : \mathcal{P}& \longrightarrow \CC \\
 \mathfrak{p} &\longmapsto - \int_{D_\mathfrak{p}} \Omega
\end{align*}
obtained by integrating the holomorphic symplectic form 
$\Omega$.
In this case, critical points are intersection points of $L_1$ and $L_2$, 
gradient flow lines are holomorphic curves in $M$ for hyperkähler rotated complex structures and the $W$-perturbed holomorphic curve equation is a 3d Fueter equation for maps $\RR^2 \times [0,1] \rightarrow M$, see \S \ref{sec:details_HFT} for details. 

Applying formally the constructions of \cite{GMWinfrared}, one should be able to define:
\begin{enumerate}
\item a vector space $H(p,p')$ for every pair 
of intersection points $p,p' \in L_1 \cap L_2$, obtained as the cohomology of a complex generated by holomorphic curves and with differential given by counts of Fueter maps. 
\item a $A_\infty$-category 
$C(L_1,L_2)$ obtained as the ``Fukaya-Seidel category of $(\mathcal{P},W)$". 
\end{enumerate}
The collection of all holomorphic Lagrangian submanifolds in 
$(M,\Omega)$ should form a $2$-category\footnote{More precisely a stable $(\infty,2)$-category.} $\mathrm{Ft}(M,\Omega)$, that we refer to as the 
\emph{Fueter 2-category} of $(M, \Omega)$, with the categories $C(L_1,L_2)$ as Hom-categories:
\[ C(L_1,L_2)=\Hom_{\mathrm{Ft}(M, \Omega)}(L_1,L_2) \,.\]
The holomorphic action functional is in general multivalued, and so considering a cover of the space of paths $\mathcal{P}$
is necessary in general for the definition of the vector spaces $H(p,p')$ and of the categories $C(L_1,L_2)$. This point is essential in the formulation of the relation with Donaldson-Thomas invariants in the following section.

As reviewed in more details in \S\ref{sec:related_works}, various aspects of holomorphic Floer theory are also discussed by Doan-Rezchikov \cite{DR}, 
Khan \cite{khan2021categorical} and Kontsevich-Soibelman \cite{KS_HFT}.
The main contribution of the present paper is to formulate several expected properties of holomorphic Floer theory, and in particular its conjectural relation with DT theory. 

\subsection{DT theory and holomorphic Floer theory}
\label{sec:dt_intro}

\subsubsection{DT theory and complex Morse theory}
DT invariants are counts of stable objects in a 3-dimensional Calabi--Yau triangulated category, such as for example the derived category of coherent sheaves or the Fukaya category of a Calabi-Yau 3-fold. As a function of the stability, DT invariants jump across real codimension one loci in a way determined by a universal wall-crossing formula \cite{JS, kontsevich2008stability}. 

Many of the algebraic structures underlying DT invariants and their wall-crossing are formally similar to the ones appearing in complex Morse theory for a holomorphic Morse function on a Kähler manifold. For example, the wall-crossing formula can be viewed as an iso-Stokes condition \cite{ BTL, Jhol}, and the DT invariants naturally define Riemann-Hilbert problems and geometric structures formally similar to Frobenius manifold structures \cite{BRH, BgeomDT}. In this analogy, DT invariants play the role of counts of gradient flow lines in complex Morse theory.

Whereas complex Morse theory of a holomorphic Morse function with $n$ critical points leads to algebraic structures based on the finite-dimensional Lie algebra
$\mathfrak{gl}_n$ of $n \times n$ matrices, the algebraic structures coming from DT theory are based on the infinite-dimensional Lie algebra $\mathfrak{g}_{DT}$ of vector fields preserving a holomorphic symplectic form on a complex torus. 
Nevertheless, it is natural to ask if this relation between DT theory and complex Morse theory theory is deeper than a mere analogy:

\begin{question} \label{q1}
Is DT theory actually an example of complex Morse theory? 
\end{question}
\vspace{0.1cm}
\noindent
Answering Question \ref{q1} means finding a Kähler manifold $\mathcal{P}$ with a holomorphic function
$W \colon \mathcal{P} \rightarrow \CC$ such that a formal application of complex Morse theory to $(\mathcal{P},W)$ reproduces DT theory. 
In this context, a natural puzzle is to understand the origin of the Lie algebra $\mathfrak{g}_{DT}$: as $\mathfrak{g}_{DT} \neq \mathfrak{gl}_n$, a necessary condition is that $W$ has infinitely many critical points.  

Question \ref{q1} admits a natural reformulation in terms of physics. DT invariants can often be interpreted as counts of BPS states in  
4-dimensional 
$\mathcal{N}=2$ supersymmetric
quantum field theories or string theory compactifications. On the other hand, a holomorphic Morse function on a Kähler manifold
defines a Landau-Ginzburg  
2-dimensional $\mathcal{N}=(2,2)$ quantum field theory. In this context, the analogy \cite{GMN1} is that the wall-crossing formula for the BPS states of 4-dimensional $\mathcal{N}=2$ theories is formally similar to the Cecotti-Vafa wall-crossing formula for the BPS states in 2-dimensional 
$\mathcal{N}=(2,2)$
theories \cite{CV}. Therefore, answering Question \ref{q1} seems to require the construction of an appropriate 2-dimensional $\mathcal{N}=(2,2)$ theory from a 4-dimensional $\mathcal{N}=2$ theory.

\subsubsection{DT invariants and holomorphic Floer theory of complex integrable systems}
\label{sec_intro_dt}

We will argue that holomorphic Floer theory gives a positive answer to Question \ref{q1} in the particular case of DT invariants describing the BPS spectrum of a $4$-dimensional $\mathcal{N}=2$ field theory (without dynamical gravity). This includes for example string compactifications on non-compact Calabi-Yau 3-folds, but not string compactifications on compact Calabi--Yau 3-folds.

Let $\mathcal{T}$ be a 4-dimensional 
$\mathcal{N}=2$
field theory. The low energy physics of $\mathcal{T}$ is controlled by its Seiberg-Witten complex integrable system 
\[ \pi \colon M \longrightarrow B\,.\] 
Here $B$ is the Coulomb branch of $\mathcal{T}$ on $\RR^4$ and $M$ is the Coulomb branch of 
$\mathcal{T}$ on $\RR^3 \times S^1$. The space $M$ is hyperkähler and  the holomorphic symplectic form $\Omega_I$ for the complex structure $I$ in which $\pi$ is holomorphic. The fibers $F_b:= \pi^{-1}(b)$ over points $b \in B$ in the complement of the discriminant locus $\Delta$ of $\pi$
are abelian varieties and holomorphic Lagrangian submanifolds of $(M,\Omega_I)$.
The lattice of charges for states in the vacuum $b \in B \setminus \Delta$ is 
\[ \Gamma_b := \pi_2(M,F_b) \]
and for every $\gamma \in \Gamma_b$, we have a space 
$BPS_\gamma^b$ of BPS states of charge $\gamma$
\cite{MR1293681, MR1306869, SW_3d, GMN1, GMN2}.

In many cases, there is an associated 3-dimensional Calabi-Yau (CY3) triangulated category $\mathcal{C}$, the base 
$B \setminus \Delta$ maps to the space of Bridgeland stability conditions on $\mathcal{C}$ \cite{MR2373143}, the central charge of the stability corresponding to a point $b\in B\setminus \Delta$ is given by 
\[ Z_\gamma(b)=\int_\gamma \Omega_I \,,\]
and the spaces of BPS states $BPS_\gamma^b$
are mathematically realized as a cohomological DT invariants of $\mathcal{C}$ counting $b$-stable objects of class $\gamma$. A concrete example is provided by class $\mathcal{S}$ theories of type $A_1$ defined by a Riemann surface $C$ (possibly with punctures)\cite{GMN1, GMN2}. In this case, $M$ is a moduli space of 
$SL(2,\CC)$ Higgs bundles on $C$, $B$ is a space of quadratic differentials on $C$, and $\pi$ is the Hitchin map. Moreover,  $\mathcal{C}$
is a 3-dimensional Calabi-Yau triangulated category described by a quiver with potential defined in terms of ideal triangulations of $C$ \cite{BS}, and can also be viewed as the Fukaya category of a non-compact Calabi-Yau 3-fold constructed from $C$ \cite{Smith1}.

Our main claim is that the DT theory of $\mathcal{C}$ for the stability $u$ is equivalent to the holomorphic Floer theory of $M$ for the pair of holomorphic Lagrangian submanifolds given by the torus fiber $F_b :=\pi^{-1}(b)$ and a section $S$ of $\pi$.
Physically, $S \subset M$ is a holomorphic Lagrangian section of 
$\pi$ defined by putting $\mathcal{T}$ on a cigar geometry, see \S \ref{sec_main_hft_conj} for details. For example, $S$ is the Hitchin section when $M$ is a moduli space of Higgs bundles.
In other words, our proposed answer to Question \ref{q1} is positive in this setting and one should recover DT theory by applying complex Morse theory to the infinite dimensional Kähler manifold
$\mathcal{P}$ of paths between $S$ and $F_b$, equipped with the holomorphic action functional $W: \mathfrak{p} \mapsto -\int_{D_{\mathfrak{p}}} \Omega_I$.

As $S$ is a section of $\pi$ and $F_b$ is a fiber, the intersection 
$S \cap F_b$ consists of a single point. However,
$\mathcal{P}$ is not simply connected: as we will show in \S\ref{sec_main_hft_conj}, we have 
\[ \pi_1(\mathcal{P}) \simeq \pi_2(M,F_b)= \Gamma_b \,.\]
The holomorphic action functional $W$ becomes a single-valued function on the universal cover $\widetilde{\cP}$ of $\cP$, and has infinitely many critical points which form a torsor under $\Gamma_b$. 
Fixing a reference critical point $0$, we identify the set of critical points with $\Gamma_b$. We can now state our main conjecture relating DT theory of the CY3 category 
$\mathcal{C}$ and the holomorphic Floer theory of the Seiberg-Witten integrable system $(M,\Omega_I)$.

\begin{conjecture} \label{conj_intro}
For every $b \in B \setminus \Delta$ and $\gamma \in \Gamma_b$, the cohomological DT invariant $BPS_\gamma^b$ is isomorphic to the vector space $H(0,\gamma)$
attached by holomorphic Floer theory of the pair of holomorphic Lagrangian submanifolds $S$ and $F_b$ in $(M,\Omega_I)$
to the two critical points $0$ and $\gamma$
of the holomorphic action functional:
\[ BPS_\gamma^b \simeq H(0,\gamma) \,.\]
\end{conjecture}

In \S \ref{sec_derivation}, we give a physics argument in favor of Conjecture 
\ref{conj_intro}. In particular,
we will obtain an answer to the physics reformulation
of Question \ref{q1}: the relevant
2d $\mathcal{N}=(2,2)$ theory is obtained by compactification of 
$\mathcal{T}$ on a cigar geometry with appropriate boundary conditions.
In \S\ref{sec_lines}, we present further conjectures involving the 
$A_\infty$-category 
\[ C(S,F_b)=\Hom_{\mathrm{Ft}(M,\Omega_I)}(S,F_b)\] 
attached by holomorphic Floer theory to the pair of
holomorphic Lagrangian submanifolds $S$ and $F_b$.

\subsection{Related works}
\label{sec:related_works}
While thinking about Question \ref{q1} in 2017, the author stumbled upon the idea to consider the Fukaya-Seidel of the holomorphic action functional and did the exercise to derive the corresponding Fueter equation.
Recently, Doan-Rezchikov \cite{DR} independently derived the 2-categorical picture of holomorphic Floer theory. More importantly, they started the analysis of the compactness properties of the Fueter equation. 
Any attempt to rigorously discuss the conjectures of our paper should probably start with their work.  

In his thesis
\cite[\S 1.1.1]{khan2021categorical} Khan also considered the holomorphic area functional and derived the Fueter equation as the
corresponding $\zeta$-instanton equation. Another important part of this work is a systematical study of the 2d categorical wall-crossing \cite{khan2020categorical}] and of the 2d categorical wall-crossing with ``twisted masses" \cite[\S 3]{khan2021categorical}, which applies in particular to multivalued Landau-Ginzburg potential. In particular, formal aspects of the 4d wall-crossing appear in this context. From the point of view of our Conjecture \ref{conj_intro}, this is not surprizing because we claim that the 4d wall-crossing is an example of 2d wall-crossing for the multivalued holomorphic action functional.  
In particular, applying \cite[\S 3]{khan2021categorical} to the holomorphic area functional should help to make more precise the algebraic structures involved in Conjecture 
\ref{conj_intro}.

The 2-categories attached to holomorphic symplectic varieties conjecturally enter in 3d mirror symmetry. In this context, the B-side is the Rozansky-Witten 2-category, based on the category of matrix factorizations. The A-side should be a gauged version of the Fueter category, based on the Fukaya-Seidel category \cite{gammage2022perverse2, gammage2022perverse}: concrete examples of the A-side in the literature rely on explicit algebraic models of the 2-category and it is an interesting problem to clarify the relation with the geometric model given by the Fueter 2-category. 

Finally, Kontsevich-Soibelman have an ongoing work on holomorphic Floer theory \cite{KS_HFT}, not focused on the 2-categorical aspects related to the Fueter equation, but rather on the exploration of the 1-categorical level in many directions that we do not touch upon, such as the relation with deformation quantization or a general theory of resurgence, see \S\ref{sec_rw}. 

\subsection{Plan of the paper}
After an introduction to finite-dimensional complex Morse theory in \S\ref{sec:fd_complex_morse}, we introduce holomorphic Floer theory in \S\ref{sec:details_HFT}. Conjectures relating holomorphic Floer theory and the structure of 2d $\mathcal{N}=(2,2)$ theories are discussed in \S\ref{sec_2d}. Finally, the relation between holomorphic Floer theory and DT invariants is the subject of \S\ref{sec_dt}. Whereas the conjectures in \S\ref{sec:details_HFT} are stated entirely mathematically, conjectures in \S\ref{sec_2d}-\S\ref{sec_dt} are stated in physical terms for maximal generality, even if they have mathematically well-formulated special cases.

\subsection{Acknowledgment}
I thank Tom Bridgeland and Richard Thomas for discussions at an early stage in 2017. I also thank the organizers of the String Math 2022 conference at the
University of Warsaw, 
July 11-15, 2022,
and of the MSRI Workshop New Four-Dimensional Gauge Theories, October 24-28 2022,
for invitation to give talks where this work was presented.

\section{Complex Morse theory}
\label{sec:fd_complex_morse}

As announced in \S \ref{sec:intro_HFT}, holomorphic Floer theory should be viewed as an example of infinite-dimensional 
complex Morse theory. In this section, we first briefly review aspects of finite-dimensional complex Morse theory. 

%before giving more details on the expected construction of holomorphic Floer theory in \S \ref{sec:details_HFT}.

\subsection{Solitons and instantons}
Let $X$ be a Kähler manifold, with complex structure $J$ and Kähler form $\omega$, such that $c_1(X)=0$.
Let $W \colon X \rightarrow \CC$
be a holomorphic function which is Morse, that is with non-degenerate critical points, and which has distinct critical values. In this section, we assume that $\omega=d\lambda$ is exact and that the set of critical points of $W$ is finite. We will have to eventually relax these assumptions for the infinite-dimensional set-up of holomorphic Floer theory, see \S \ref{sec_cat-pairs_lag}.

Following Haydys \cite{Hayd} and Gaiotto-Moore-Witten \cite{GMWinfrared}, complex Morse theory for $(X,W)$ is based on the following two sets of equations, parametrized by a phase $\zeta \in \CC$, $|\zeta|=1$:
\begin{enumerate}
    \item for a path 
    \begin{align}
        u \colon \RR &\longrightarrow X \\
\nonumber        x &\longmapsto u(x) \,,
    \end{align}
the equation 
\begin{equation} \label{eq_soliton}
    \partial_x u + \mathrm{grad} \left(\Rea \left(\zeta^{-1}W(u)\right)\right)=0
\end{equation}
is the gradient flow line equation for the Morse function $\Rea (\zeta^{-1} W)$, or equivalently the Hamiltonian flow equation for $\Ima (\zeta^{-1} W)$. 
    \item for a plane
\begin{align}
    u \colon \RR^2 &\longrightarrow X \\
    \nonumber (t,x) &\longmapsto u(t,x) \,,
\end{align}
the equation 
\begin{equation} \label{eq_instanton}
    \partial_t u + J(u) \left(\partial_x u +  \mathrm{grad} \left(\Rea \left(\zeta^{-1} W(u)\right)\right) \right)=0
\end{equation}
is the holomorphic curve equation perturbed by the Hamiltonian $\Ima(\zeta^{-1} W)$. 
\end{enumerate}

Equations \eqref{eq_soliton} and \eqref{eq_instanton} are respectively the critical point equation and the gradient flow line equation for the functional 
\[ u \mapsto \int_\RR \left(u^* \lambda - \Ima(\zeta^{-1} W)dx \right)\]
on the space of paths $\mathrm{Maps}(\RR,X)=\{u \colon \RR \rightarrow X\} $. In \cite[\S 11.1]{GMWinfrared}, 
\eqref{eq_soliton} and \eqref{eq_instanton} are respectively called the \emph{$\zeta$-soliton equation} and the \emph{$\zeta$-instanton equation}\footnote{Our convention for 
$\zeta$ agrees with \cite{KKS}, differs from 
\cite{khan2020categorical} by a factor of $-1$,
and differs from 
\cite{GMWinfrared} by a factor of $i=\sqrt{-1}$.} of the 2-dimensional
$\mathcal{N}=(2,2)$
massive Landau-Ginzburg theory defined by $(X,W)$.\footnote{In the context of FJRW invariants, \eqref{eq_instanton} is sometimes called the Witten equation \cite{FJR}.} Following this terminology, we will use \emph{$\zeta$-soliton} for a solution of \eqref{eq_soliton} and \emph{$\zeta$-instanton} for a solution of \eqref{eq_instanton}.

\subsection{2d BPS states} \label{sec_2d_bps}
Let $\VV$ be the set of critical points of $W$.
Complex Morse theory attaches a vector space $H(p,p')$, referred to as the  \emph{space of 2d BPS states}, to every pair $p,p' \in \VV$ of distinct critical points. The 
vector space $H(p,p')$ should be defined as the cohomology of a complex $(C(p,p'),d)$, where 
\begin{enumerate}
    \item $C(p,p')$ is the vector space generated by 
    paths \begin{align*}
        u_i \colon \RR &\longrightarrow X \\
        x &\longmapsto u_i(x) \,,
    \end{align*} 
    considered up to $\RR$-translation, asymptotic to $p$ and $p'$,
    \[ \lim_{x \rightarrow -\infty} u_i(x)=p\,,\,\,
    \lim_{x \rightarrow +\infty} u_i(x) =p'\,,\]
    and which are $\zeta_{p,p'}$-solitons, that is solutions of the gradient flow line equation \eqref{eq_soliton} with 
    \[ \zeta:= \zeta_{p,p'}:=\frac{W(p)-W(p')}{|W(p)-W(p')|}\,.\]
    \item $d \colon C(p,p') \rightarrow C(p,p')$ is the differential defined by 
    \[ d u_i = \sum_j n_{ij} u_j \,,\]
    where $n_{ij}$ is a signed count of planes 
    \begin{align*}
    u \colon \RR^2 &\longrightarrow X \\
    (t,x) &\longmapsto u(t,x) \,,
        \end{align*}
    asymptotic to $u_i$ and $u_j$,
    \[ \lim_{t \rightarrow -\infty}u(t,x)=u_i(x)\,,\,\, 
    \lim_{t \rightarrow +\infty} u(t,x)=u_j(x)\,,\]
    and which are $\zeta_{p,p'}$-instantons, that is
    solutions of \eqref{eq_instanton} with $\zeta=\zeta_{p,p'}$, and considered up to the action of $\RR$ induced by translation of the $t$-variable, see Figure \ref{fig1}.
\end{enumerate}
In the terminology of \cite{GMWinfrared}, $H(p,p')$ is the space of 2d BPS states connecting the vacua $p$ and $p'$
of the 2-dimensional $\mathcal{N}=(2,2)$
massive Landau-Ginzburg theory defined by $(X,W)$.

\begin{figure}[ht!]
\centering
\includegraphics[width=50mm]{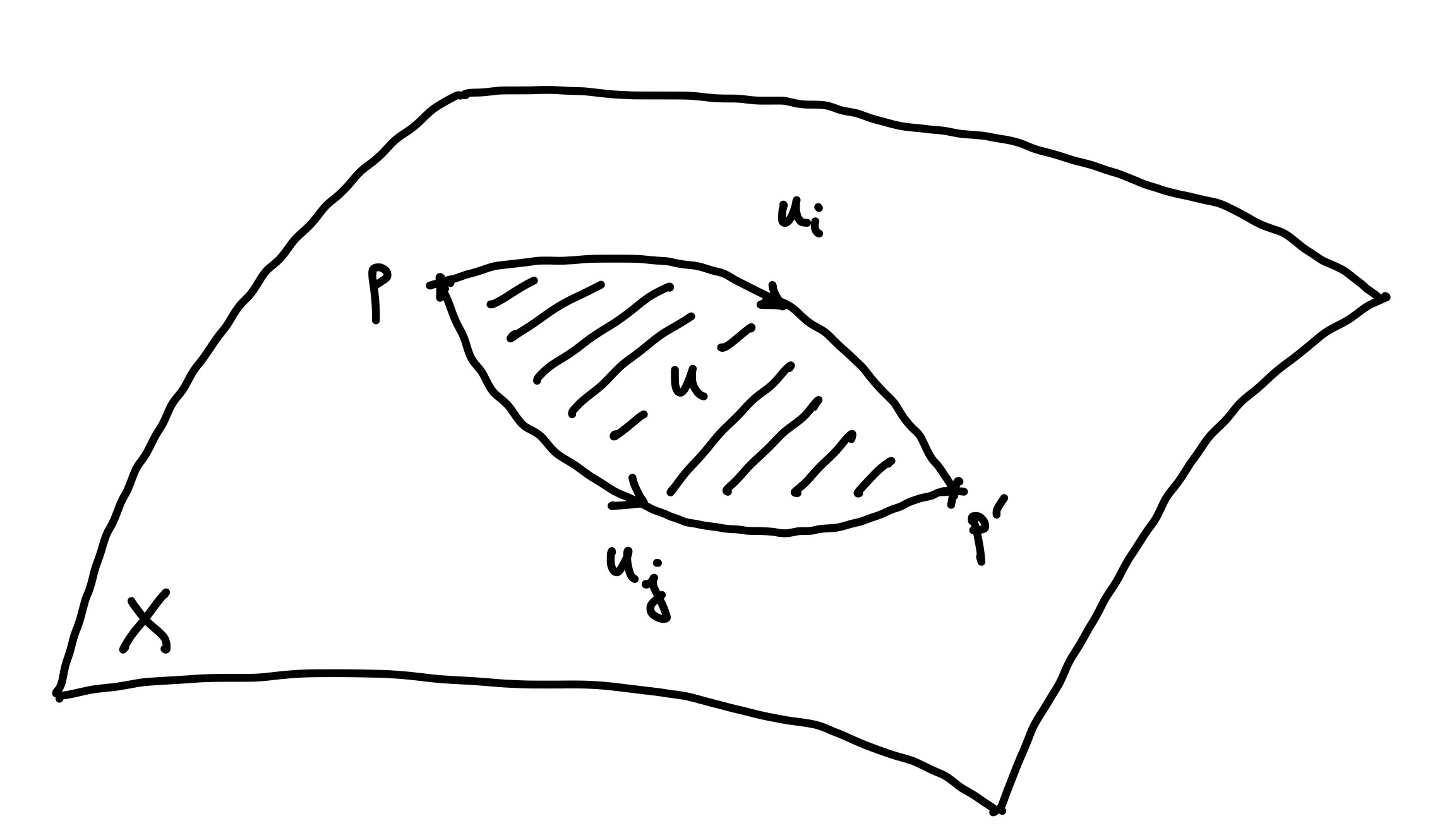}
\caption{Contribution to the differential of a $\zeta_{p,p'}$-instantons $u$ asymptotic to gradient flow lines $u_i$ and $u_j$ between critical points $p$ and $p'$. \label{fig1}}
\end{figure}

\subsection{Fukaya-Seidel category} \label{sec_FS_category}
Complex Morse theory also attaches a $A_\infty$-category $FS(X,W)$ to $(X,W)$, which is conjecturally equivalent to the Fukaya-Seidel category of $(X,W)$. We describe in a few words the web-based construction of Gaiotto-Moore-Witten \cite{GMWinfrared}, reformulated in the dual language of polygons
by Kapranov-Kontsevich-Soibelman \cite{KKS} (see also \cite{kapranov2020perverse}). 
We denote by $\CC_w$ the copy of $\CC$ where $W$ takes its values.
Let $\CC_w \cup S^1_\infty$ be the compactification of $\CC_w$
obtained by adding a circle of oriented directions at infinity. We fix a phase $\zeta \in \CC$, $
|\zeta|=1$, and we denote by $w_\infty \in S^1_\infty$ the corresponding point at infinity. 
We think of $w_\infty$ as being the limit $R \rightarrow +\infty$ of the point $R \zeta_\infty \in \CC_w$, with $R>0$.
Finally, we assume that $\zeta \neq \zeta_{p,p'}$
for every $p,p' \in \mathbb{V}$.

For every pair $p$ and $p'$ of distinct critical points, we define the complex
\begin{equation} 
\label{eq_complex}
R_{p,p'}:= \bigoplus_{Q \in \mathscr{Q}_{p,p'} } R_Q \,, \end{equation}
where the direct sum is over the set $\mathscr{Q}_{p,p'}$ 
of all the convex polygons $Q$ in $\CC_w \cup S^1_\infty$ with vertices in clockwise order given by $w_\infty, W(p), W(p_1), \dots, W(p_k), W(p')$, where $p_1, \dots, p_k \in \mathbb{V}$ are distinct critical points, also distinct from $p$ and $p'$, see Figure \ref{fig2}, and where $R_Q$ is the complex
\begin{equation}\label{eq_RQ}
R_Q := C(p,p_1)\otimes C(p_1,p_2) \otimes \cdots \otimes C(p_{k},p') \,.\end{equation} 
Moreover, for every critical point $p$, we set $R_{p,p}=\CC$ with trivial differential.

\begin{figure}[ht!]
\centering
\includegraphics[width=40mm]{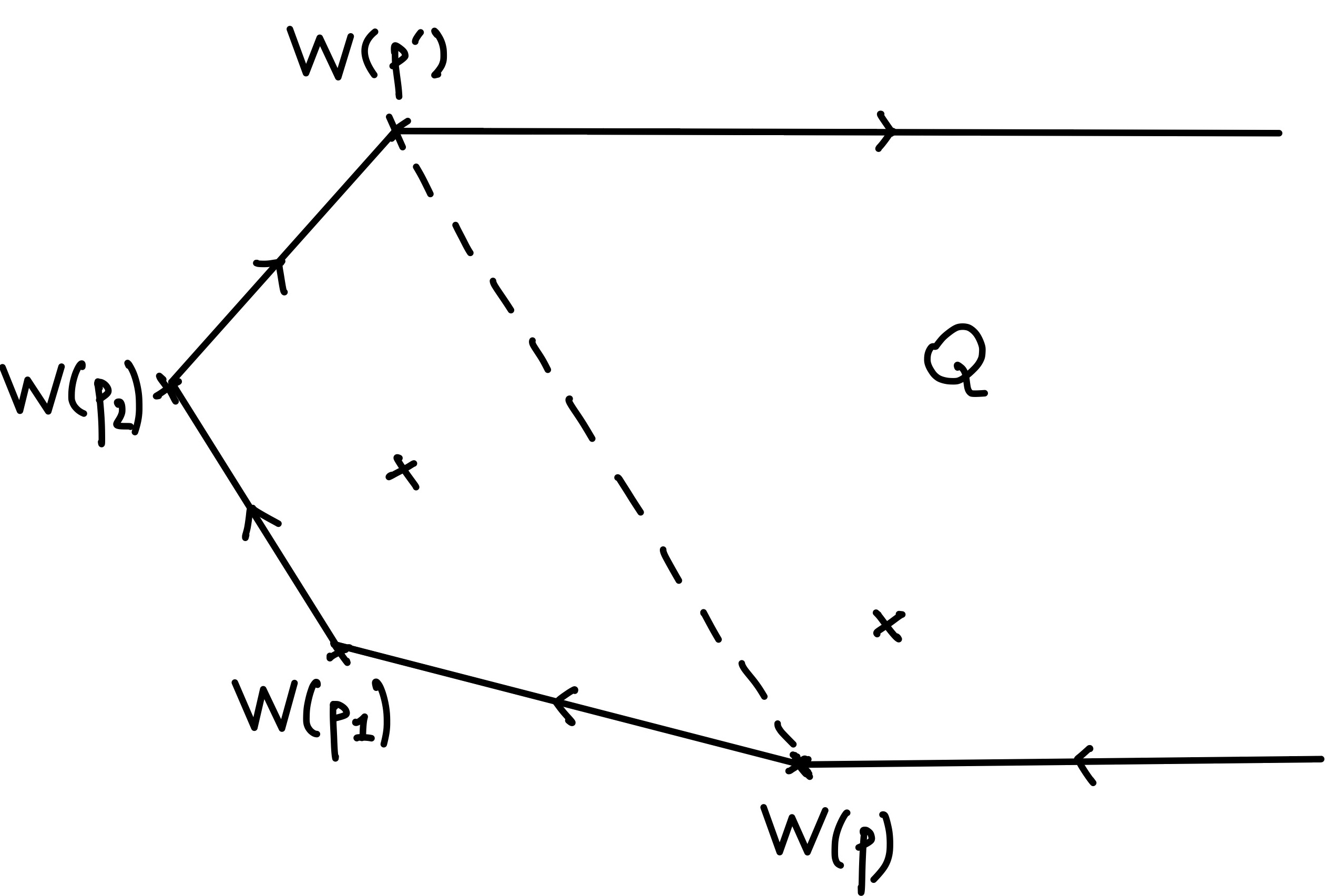}
\caption{A polygon $Q\in \mathscr{Q}_{p,p'}$ with $\zeta=1$. \label{fig2}}
\end{figure}

Let $R$ be the dg-category with objects $T_p(\zeta)$ indexed by 
the critical points $p$, Hom complexes given by 
\[ \Hom_R(T_p(\zeta),T_{p'}(\zeta))=R_{p,p'} \,,\]
and where the composition maps 
\begin{align*} R_{p,p'} \otimes R_{p',p''}& \longrightarrow R_{p,p''}  \\
a \otimes b& \longmapsto ab
\end{align*}
are defined as follows: for 
$Q \in \mathscr{Q}_{p,p'}$, $Q' \in \mathscr{Q}_{p',p''}$, $a \in R_Q$ and $b\in R_{Q'}$, $ab \in R_{Q \cup Q'}$ is the natural concatenation of $a$ and $b$ if the polygon $Q \cup Q'$ obtained by gluing $Q$ and $Q'$ along their common edge $[p',w_\infty]$ is convex, and $ab=0$ otherwise. In what follows, we view the
dg-category $R$ as an example of $A_\infty$-category with trivial higher operations $(m_k)_{k \geq 3}$. By construction, $\Hom_R(T_p(\zeta),T_p'(\zeta))=R_{p,p'}=0$ if the triangle with ordered vertices $w_\infty, W(p), W(p')$ is oriented counterclockwise, and so $R$ is an example of directed $A_\infty$-category.

To obtain a $A_\infty$-category with non-trivial higher operations, one considers the following counting problems. 
Let $\mathcal{P}$ be the set of convex polygons $P$ in $\CC_w$
with vertices in clockwise order given by 
$W(p_1), \dots, W(p_k)$, where $p_1, \dots, p_k \in \mathbb{V}$ are distinct critical points. Given such a polygon $P$, we define a map
\[ \mu_P \colon C(p_1,p_2) \otimes \cdots \otimes C(p_{k-1},p_k) 
\longrightarrow \ZZ \]
as follows: given $\zeta_{p_i,p_{i+1}}$-solitons generators $\alpha_i$ of $C(p_i,p_{i+1})$ for $1 \leq i \leq k-1$,
\[ \mu_P(\alpha_1, \dots,\alpha_{k-1})\] 
is a count of $\zeta$-instantons
\[ u \colon \CC \longrightarrow X \]
such that the restriction of $u$ to a direction perpendicular to $\RR_{\geq 0}(W(p_{i+1})-W(p_i))$ asymptotically approaches the $\zeta_{p_{i+1},p_i}$-soliton $\alpha_i$ rotated by a phase $\zeta \zeta_{p_i,p_{i+1}}^{-1}$. 

The main result of \cite{GMWinfrared} is that 
$\sum_{P \in \mathscr{P}} \mu_P$ naturally defines a Maurer-Cartan element for $R$, and so one can construct a new $A_\infty$-algebra $\tilde{R}$ by deforming $R$ by this Maurer-Cartan element.
Explicitly, higher $A_\infty$-operations of
$\tilde{R}$ are defined by deforming the naive concatenation of polygons $Q \in \mathscr{Q}_{p,p'}$ by terms indexed by regular polyhedral decompositions of the concatenated polygons, and where each finite piece $P$ of a polyhedral decomposition contributes a factor $\mu_P$. Passing to the category of twisted complexes of $\tilde{R}$-modules, we obtain a triangulated $A_\infty$-category $FS(X,W)$ which is the model of \cite{GMWinfrared, KKS} for the Fukaya-Seidel of $(X,W)$.

The categories $R$ and $\tilde{R}$ depend on the choice of the phase $\zeta$, but the triangulated category $FS(X,W)$ is independent of these choices.
By construction, for a fixed choice of $\zeta$, we have objects $(T_p(\zeta))_{p \in \mathbb{V}}$ which form an exceptional collection of $FS(X,W)$ when ordered according to the values of $\Ima (\zeta^{-1}W(p))$. When $\zeta$ crosses a value $\zeta_{p,p'}$, this exceptional collection changes by a mutation. Comparing with the original construction of the Fukaya-Seidel category, $T_p(\zeta)$ corresponds to the Lefschetz thimble emanating from $p$, which consists of the gradient flow lines starting at $p$ and going to $\Rea(\zeta^{-1}W)=-\infty$, and whose image by $W$ is the half-line $W(p)-\RR_{\geq 0}\zeta$
in $\CC_w$.

Following the terminology of \cite{kapranov2020perverse}, the above construction of the Fukaya-Seidel category of \cite{GMWinfrared, KKS} follows a \emph{rectilinear} approach, based on $\zeta_{p,p'}$-solitons whose images by $W$ are the straight line segments $[W(p),W(p')]$. In \cite{Hayd}, a \emph{curvilinear} approach is presented instead, based on a system of possibly curvy lines connecting the critical values $W(p)$ to a given base point $-R \zeta$ with $R>>0$. This approach is in a way more direct, as it avoids the non-trivial web/convex geometry of \cite{GMWinfrared, KKS}, but is also less explicit because it relies on versions of $\zeta$-solitons and $\zeta$-instanton equations where $\zeta$ is $x$ or $(t,x)$-dependent. Intuitively, one goes from the curvilinear approach of \cite{Hayd}
to the rectilinear approach of \cite{GMWinfrared, KKS} by starting with a possibly curvy line connecting $W(p)$, $W(p')$ passing through $-R\zeta$, and then trying to straighten it while relaxing the condition to pass through $-R \zeta$: the convex polygonal lines connecting $W(p)$, $W(p')$ appear upon hitting critical values, see Figure \ref{fig3}.

\begin{figure}[ht!]
\centering
\includegraphics[width=40mm]{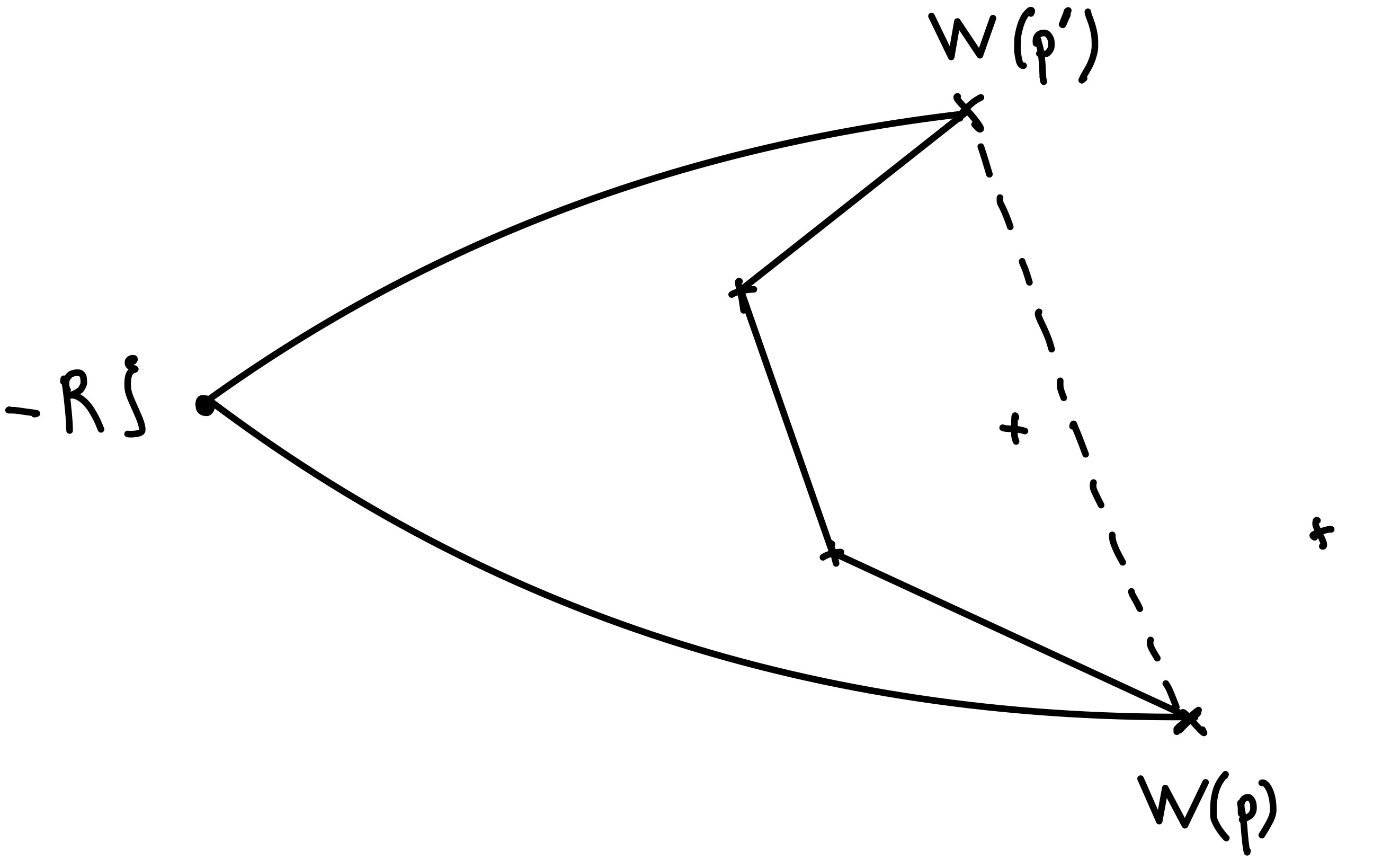}
\caption{Comparison of the curvilinear and rectilinear approaches to complex Morse theory. \label{fig3}}
\end{figure}

%\subsubsection{Seidel's original definition of the Fukaya-Seidel category}

%Fix a base point $z_0 \in \C$ note lying on any line connecting two critical values. Order the critical points $p_1, \dots, p_n \in \VV$ such that the sequence $\mathrm{Arg}(W(p_i)-z_0)$ is decreasing in $i$. For every $1 \leq i \leq n$, fix a path from $W(p_i)$ to $z_0$ avoiding the other critical values, and denote by $L_j \subset W^{-1}(z_0)$ the corresponding Lagrangian vanishing cycle.

\subsection{Infinite-dimensional examples}
\label{sec_inf_ex}
Complex Morse theory can be formally applied to pairs $(X,W)$ where $X$ is infinite-dimensional. 
For example, one can take for $X$ the space of $G_\CC$-connections on a 3-manifold $M$, where $G_\CC$ is a complex reductive group, and for $W$ the complex-valued Chern-Simons functional \cite{Hayd, Wanalytic}\cite[\S 4.2]{khan2021categorical}. In this case,
\begin{enumerate}
    \item critical points of $W$ are flat $G_\CC$-connections on $M$,
    \item  the $\zeta$-soliton equation \eqref{eq_soliton}
    is the Kapustin-Witten equation in 4-dimensional gauge theory on $M \times \RR$,
    \item the $\zeta$-instanton equation \eqref{eq_instanton}
    is the Haydys-Witten equation in 5-dimensional gauge theory on $M \times \RR^2$.
\end{enumerate}
Another example is obtained by taking
for $X$ the space of $(0,1)$-connections on a vector bundle $V$ on a Calabi-Yau 3-fold and for $W$ the holomorphic Chern-Simons functional \cite{DTgauge, DS, Hayd}. 
In this case,
\begin{enumerate}
\item critical points of $W$ are holomorphic structures on $V$,
\item the $\zeta$-soliton equation \eqref{eq_soliton} is the equation for $G_2$-instantons in 7-dimensional gauge theory on $M \times \RR$,
\item
the $\zeta$-instanton equation \eqref{eq_instanton}
is the equation for 
$Spin_7$-instantons in 8-dimensional gauge theory
on $X \times \RR^2$.
\end{enumerate}
As explained in the following section \S\ref{sec:details_HFT}, holomorphic Floer theory is another example obtained when $X$ is the space of paths stretched between two holomorphic Lagrangian submanifolds $L_1$ and $L_2$ of a holomorphic symplectic manifold $(M,\Omega)$, and $W$ is the holomorphic action functional. In this case,
\begin{enumerate}
\item critical points of $W$ are intersection points of $L_1$ and $L_2$,
\item the $\zeta$-soliton equation \eqref{eq_soliton} is the equation for holomorphic curves $\RR \times [0,1] \rightarrow M$ with respect to hyperk\"ahler rotated complex structures (see Lemma \ref{lem_soliton}),
\item
the $\zeta$-instanton equation \eqref{eq_instanton}
is the 3d Fueter equation for maps $\RR^2 \times [0,1] \rightarrow M$ (see Lemma \ref{lem_instanton}).
\end{enumerate}

\section{Holomorphic Floer theory}
\label{sec:details_HFT}

In this section, we introduce holomorphic Floer theory as an example of infinite dimensional complex Morse theory applied to the holomorphic action functional on the space of paths connecting two holomorphic Lagrangian submanifolds in a holomorphic symplectic variety. We only describe a conjectural picture ignoring all analytic difficulties. We refer to the work of Doan-Rezchikov \cite{DR}
for first steps towards solving these analytic problems.

\subsection{Holomorphic symplectic geometry}
\label{sec_hsg}
Let $M$ be a holomorphic symplectic manifold, with complex structure $I$ and holomorphic symplectic form 
$\Omega_I$. We fix a compatible hyperk\"ahler structure on $M$, that is a Riemannian metric $g$ and complex structures $J$ and $K$ such that:
\begin{enumerate}
    \item $I$, $J$, $K$ form a quaternionic triple, that is $IJK=-1$.
    \item $g$ is K\"ahler with respect to $I$, $J$, $K$. We denote by 
    \[ \omega_I:=g(I-,-)\,,\,\, \omega_J:=g(J-,-)\,,\,\, \omega_K:=g(K-,-)\] the corresponding K\"ahler forms.
    \item $\Omega_I= \omega_J+i \omega_K$.
\end{enumerate}

Given a hyperk\"ahler structure as above, the metric $g$ is K\"ahler with respect to any complex structure in the \emph{twistor sphere} spanned by $I$, $J$, $K$, that is any complex structure of the form $aI+bJ+cK$ with $a,b,c \in \RR$ such that $a^2+b^2+c^2=1$. For any phase $\zeta \in \CC$, $|\zeta|=1$, we denote \[ J_\zeta := (\Rea\, \zeta) J + (\Ima\, \zeta) K\,.\] 
If $I$ and $-I$ are viewed as respectively the North and South poles of the twistor sphere, then the complex structures of the form $J_\zeta$ with $|\zeta|=1$ are exactly the complex structures on the Equator of the twistor sphere.
The K\"ahler form of the metric $g$ for the complex structure $J_\zeta$ is 
\begin{equation} \label{eq_omega_zeta}
\omega_\zeta:= \Rea (\zeta^{-1} \Omega_I)=(\Rea\,\zeta)\omega_J+(\Ima\,\zeta)\omega_K\,.
\end{equation}

\subsection{The holomorphic action functional} 
\label{sec_hol_action}
Let $L_1$ and $L_2$ be two connected holomorphic Lagrangian submanifolds of $(M, I, \Omega_I)$, that is $L_1$ and $L_2$ are $I$-holomorphic half-dimensional submanifolds
of $M$
and $\Omega_I|_{L_1}=\Omega_I|_{L_2}=0$.
Assume that the intersection $L_1 \cap L_2$ consists of finitely many transverse intersection points.

Let $\mathcal{P}$ be the infinite-dimensional manifold of smooth paths stretched between $L_1$ and $L_2$:
\[ \mathcal{P}:= \{ \mathfrak{p}: [0,1] \longrightarrow M \,|\, \mathfrak{p}(0) \in L_1 \,, \mathfrak{p}(1) \in L_2 \} \,.\]
It is naturally an infinite-dimensional K\"ahler manifold, with Riemannian metric 
\begin{equation} \label{eq_g_P}
g_{\mathcal{P}}(\delta \mathfrak{p}_1, \delta \mathfrak{p}_2) := \int_{0}^1 g(\delta \mathfrak{p}_1(y),
\delta \mathfrak{p}_2(y)) dy \,,\end{equation}
complex structure
\begin{equation}\label{eq_I}
(I_{\mathcal{P}}(\mathfrak{p})(\delta \mathfrak{p}))(y):= I(\mathfrak{p}(y))(\delta \mathfrak{p}(y))\,,
\end{equation}
and K\"ahler form
\begin{equation} \label{eq_omega}
\omega_{I,\mathcal{P}} (\delta \mathfrak{p}_1, \delta \mathfrak{p}_2) := \int_{0}^1 \omega(\delta \mathfrak{p}_1(y),
\delta \mathfrak{p}_2(y)) dy\,.\end{equation}

For every intersection point $p \in L_1 \cap L_2$, we still denote by $p$ the corresponding point of $\cP$ given by the constant path sitting at $p_0$.
As $L_1$ and $L_2$ are connected, there exists a unique connected component $\mathcal{P}_{0}$ of $\cP$
containing all these constant paths sitting at the intersection points of $L_1$ and $L_2$. For every $p,p' \in L_1 \cap L_2$, we denote by $\pi_1(p,q)$ the set of homotopy classes of paths in $\mathcal{P}_0$ from $p$ to $p'$.

We fix an intersection point $p_0 \in L_1 \cap L_2$. For every $\fop \in \cP_0$ and $w$ a path in 
$\cP_0$ connecting $p_0$ to $\fop$, that is a map
\begin{align*}
    w \colon [0,1] &\longrightarrow \cP_0 \\
    x &\longmapsto w_x 
\end{align*}
such that $w_0=p_0$ and $w_1 =\fop$,
one views $w$ as a map 
\begin{align*}
    w \colon [0,1] \times [0,1] \longrightarrow M \\
    (x,y) \longmapsto w_x(y)\,.
\end{align*}
and one defines 
\begin{equation} \label{eq_W}
W(\fop, w) := -\int_{[0,1] \times [0,1]} w^{*} \Omega_I \,.\end{equation}
As $\Omega_I$ is closed and $L_1$, $L_2$
are $I$-holomorphic Lagrangian submanifolds
($\Omega_I|_{L_1}=\Omega_I|_{L_2}=0$), 
the value
$W(\fop, w)$ only depends on the homotopy class $[w] \in \pi_1(p_0,p)$ of $w$, and so we obtain a well-defined function on the universal cover $\widetilde{\cP}_{0}$ of $\cP_{0}$:
\begin{align*} 
W \colon   \widetilde{\cP}_{0} &\longrightarrow \CC \\
    (\fop, [w]) &\longmapsto W(\fop, [w]) := -\int_{[0,1] \times [0,1]} w^{*} \Omega_I  \,,
\end{align*}
called the \emph{holomorphic action functional}. As $\Omega_I$ is a $I$-holomorphic 2-form on $M$, 
$W$ is a holomorphic function on the infinite-dimensional K\"ahler manifold $\widetilde{\cP}_{0}$ endowed with the complex structure $I_{\cP}$.

\subsection{Infinite-dimensional complex Morse theory}
\label{sec_inf_dim}
The holomorphic action functional $W$ is a holomorphic function on the
infinite-dimensional K\"ahler manifold $\widetilde{\cP}_0$.
The critical points of $W$ are of the form 
$(p,\gamma)$, where $p \in L_1 \cap L_2$ is an 
intersection point, and where $\gamma \in \pi_1(p_0,p)$ is a homotopy class of paths from $p_0$ to $p$ in $\cP_0$. These critical points are non-degenerate under the assumption that $L_1$ and $L_2$
intersect transversally.
Hence, we can apply formally complex Morse theory to $(\widetilde{\cP}_0, W)$. 

\begin{lemma} \label{lem_soliton}
For every $\zeta \in \CC$ with $|\zeta|=1$, a map
\begin{align*} u \colon  \RR
&\longrightarrow \widetilde{\cP}_0  \\
x &\longmapsto (y \mapsto u(x,y))
\end{align*}
is a solution of the $\zeta$-soliton equation for $(\widetilde{\cP}_0, W)$ if and only if 
\begin{align*}
u \colon \RR \times [0,1]& \longrightarrow M  \\
(x,y) &\longmapsto u(x,y)
\end{align*} 
is a solution of the $J_\zeta$-holomorphic curve equation 
\begin{equation} \label{eq_zeta_holomorphic}
    \partial_x u + J_\zeta(u) \partial_y u=0 \,.
\end{equation}
\end{lemma}

\begin{proof}
The $\zeta$-soliton equation \eqref{eq_soliton} for $(\widetilde{\cP}_0, W)$ is the gradient flow line equation
of the functional $\Rea (\zeta^{-1}W)$. Using \eqref{eq_omega_zeta} and \eqref{eq_W}, one finds that $\Rea (\zeta^{-1}W)$
is the action functional for the symplectic form $\omega_\zeta$, and so its gradient flow line  equation is the $J_\zeta$-holomorphic curve equation in $M$. This latter result is the well-known starting point of Floer theory \cite[(1.9-1.10)]{Flo1}, but we briefly review its derivation for convenience. Using \eqref{eq_W}, the differential
of $W$
at the point $u(x,-) :=(y \mapsto u(x,y))$ evaluated on the tangent vector field $\xi =(y \mapsto \xi(y))$ is given by 
\[ dW_{u(x,-)}(\xi)=-\int_{0}^1 \omega_\zeta (\xi(y), \partial_y u(x,y))\,.\]
As $\omega_\zeta(-,-)=-g(-,J_\zeta -)$, it follows from the definition \eqref{eq_g_P} of the metric on $\widetilde{\cP}_0$
that 
\[ (\mathrm{grad}_{u(x,-)} W)(y) = J_{\zeta}(u(x,y)) \partial_y u(x,y) \,,\]
and so the gradient flow line equation \eqref{eq_soliton} indeed reduces to the $J_\zeta$-holomorphic curve equation
\eqref{eq_zeta_holomorphic}.
\end{proof}

\begin{lemma} \label{lem_instanton}
For every $\zeta \in \CC$ with $|\zeta|=1$, a map
\begin{align*} u \colon  \RR^2
&\longrightarrow \widetilde{\cP}_0 \\
(t,x) &\longmapsto (y \mapsto u(t,x,y))
\end{align*}
is a solution of the $\zeta$-instanton equation for $(\widetilde{\cP}_0, W)$ if and only if 
\begin{align*}
u \colon \RR^2 \times [0,1]& \longrightarrow M  \\
(t,x,y) &\longmapsto u(t,x,y)
\end{align*} 
is a solution of the \emph{3d $\zeta$-Fueter equation}
\begin{equation} \label{eq_fueter}
\partial_t u +I(u)(\partial_x u +J_\zeta(u) \partial_y u)=0\,.    
\end{equation}
\end{lemma}

\begin{proof}
The result follows using the definition \eqref{eq_instanton} of the $\zeta$-instanton equation, definition 
\eqref{eq_I} of the complex structure on $\widetilde{\cP}_0$, and Lemma \ref{lem_soliton} computing the 
$\zeta$-soliton equation.
\end{proof}

\subsection{Vector spaces from pairs of intersection points}
\label{sec_2d_bps_HFT}

For every pair of intersection points $p,p' \in L_1 \cap L_2$ and homotopy classes of paths $\gamma \in \pi_1(p_0,p)$ and $\gamma'\in \pi_1(p_0,p')$ in $\mathcal{P}_0$, 
one should be able by formally applying \S \ref{sec_2d_bps}
to define a vector space
\[ H((p,\gamma),(p',\gamma'))\] 
of 2d BPS states for $(\widetilde{\cP}_0, W)$. 
Explicitly, $H((p,\gamma),(p',\gamma'))$ is defined as the cohomology of the complex $(C(p,\gamma), (p',\gamma')),d)$, where 
\begin{enumerate}
    \item $C((p,\gamma),(p',\gamma'))$ is the vector space
    generated by $J_\zeta$-holomorphic curves 
    \begin{align*}
        u_i: \RR \times [0,1] &\longrightarrow M \\
        (x,y)& \longmapsto u_i(x,y) \,,
    \end{align*}
    considered up to $\RR$-translation of $x$, 
%    where 
%    \[ \zeta=\zeta_{(p,[w]),(p',[w'])}=\frac{W_{\fop_0,\Gamma}((p,[w]))-W_{\fop_0,\Gamma}((p',[w']))}{|W_{\fop_0,\Gamma}((p,[w]))-W_{\fop_0,\Gamma}((p',[w']))|}\]
    such that 
    \[u_i(.,0) \in L_1\,, u_i(0,.)\in L_2\,,\lim_{x \rightarrow -\infty} u_i(x,.)=p\,, \lim_{x \rightarrow +\infty} u_i(x,.)=p'\,,\] 
    \[[u_i]=[\gamma' \circ \gamma^{-1}]\,,\] where $u_i$ is viewed as a path from $p$ to $p'$ in $\widetilde{\cP}_0$ and $\gamma^{-1}$ is the path obtained from $\gamma$ by reversing the orientation, and 
    \begin{equation} \label{eq_zeta}
    \zeta =\zeta_{(p,\gamma),(p',\gamma')}:=\frac{\int_{\gamma' \circ \gamma^{-1}} \Omega_I}{|\int_{\gamma' \circ \gamma^{-1}} \Omega_I|} \,,\end{equation}
    where the class $\gamma' \circ \gamma^{-1} \in \pi_1(p,p')$ 
    is viewed as defining an element in the relative homology group
    $H_2(M, L_1 \cup L_2,\ZZ)$.
    \item $d:C((p,\gamma),(p',\gamma')) \rightarrow C((p,\gamma),(q,\gamma')) $ is the differential defined by 
    \begin{equation} \label{eq_diff}
    du_i =\sum_j n_{ij} u_j \,,\end{equation}
    where $n_{ij}$ is a signed count of maps 
     \begin{align*}
    u \colon \RR^2 \times [0,1] &\longrightarrow M \\
    (t,x,y) &\longmapsto u_i(t,x,y) \,,
        \end{align*}
    such that \[u(t,x,0) \in L_1\,, u(t,x,.)\in L_2\,,\lim_{x \rightarrow -\infty} u(.,x,.)=p\,, \lim_{x \rightarrow +\infty} u(.,x,.)=p'\,,\] 
    asymptotic to $u_i$ and $u_j$,
    \[ \lim_{t \rightarrow -\infty}u(t,x,y)=u_i(x,y)\,,\,\, 
    \lim_{t \rightarrow +\infty} u(t,x)=u_j(x,y)\,,\]
    which are solutions of the 3d $\zeta_{p,p'}$-Fueter equation
\begin{equation*}
\partial_t u +I(u)(\partial_x u +J_\zeta(u) \partial_y u)=0\,,    
\end{equation*}
and considered up to $\RR$-translation of $t$.
\end{enumerate}

\begin{remark}
As $L_1$ and $L_2$ are holomorphic Lagrangian submanifolds of $(M,\Omega_I)$, they are in particular Lagrangian submanifolds of $(M,\omega_\zeta)$ for every $\zeta \in \CC$, $|\zeta|=1$, and so define natural boundary conditions for $J_\zeta$-holomorphic curves.
\end{remark}

\begin{remark}
We are making the simplifying assumption that the sum \eqref{eq_diff} is finite and we are suppressing the discussion of Novikov coefficients which would be required in general.
\end{remark}

% /

\subsection{Categories from pairs of holomorphic Lagrangian submanifolds}
\label{sec_cat-pairs_lag}

Applying formally \S \ref{sec_FS_category}, one should be able to define a $A_\infty$-category 
\[ FS(\widetilde{\cP}_{0}, W) \]
as the Fukaya-Seidel category of $(\widetilde{\cP}_{0}, W)$. Objects are pairs 
\[(p,\gamma)\] 
where $p \in L_1 \cap L_2$ and $\gamma$ is a homotopy class of paths from $\fop_0$ to $p$ in $\widetilde{\cP}_{0}$, spaces of morphisms have bases consisting of formal tensor products of $J_{\zeta_{(p,\gamma),(p',\gamma')}}$-holomorphic curves with boundary on $L_1 \cup L_2$, and $A_\infty$-morphisms are given by counts of solutions of the 3d $\zeta$-Fueter equation asymptotic to sequences of $J_{\zeta_{(p,\gamma),(p',\gamma')}}$-holomorphic curves.

To make sense of this construction, one needs to deal with the issue that the set of critical points $(p,\gamma)$ of $W$ might be infinite if the fundamental group 
$\pi_1(\cP_0)$ is infinite,
whereas we assumed in our review of complex Morse theory in \S \ref{sec:fd_complex_morse} the finiteness of the set of critical points. 
We address this issue below by considering a categorical version of the Novikov coefficients of Floer theory.

%Recall that we are fixing a base point $p_0 \in L_1 \in L_2$. We denote by $(p_0,0)$ the critical point of $W$ where $0$ is the homotopy class of the constant path sitting at $p_0$. Note that $W(p_0,0)=0$. 

First remark that by construction the chain complexes 
$C((p,\gamma),(p',\gamma'))$ only depend on the class $\gamma' 
\circ \gamma^{-1} \in \pi_1(p,p')$. For every $\alpha \in 
\pi_1(p,p')$, we denote $C(p,p',\alpha)$ for $C((p,\gamma),(p',\gamma'))$ with $\gamma' 
\circ \gamma^{-1}=\alpha$.
For every intersection points $p, p' \in L_1\cap L_2$, we denote by \[r: \pi_1(p,p') \rightarrow H_2(M,L_0 \cup L_1,\ZZ)\] 
the natural projection. By construction, $C(p,p',\alpha) \neq 0$
if and only if there exists $J_{\zeta_\alpha}$-holomorphic curves of class $\alpha$, where 
\begin{equation} \label{eq_zeta_alpha}
\zeta_\alpha = \frac{\int_{r(\alpha)} \Omega_I}{|\int_{r(\alpha)} \Omega_I|}\,.\end{equation}

\begin{lemma}
\label{lem_area}
Fix $p, p' \in L_1\cap L_2$ and $\alpha \in \pi_1(p,p')$. 
If $C$ is a $J_{\zeta_\alpha}$-holomorphic curve of class
$\alpha$, then the area of $C$ is $|\int_{r(\alpha)} \Omega_I|$.
\end{lemma}

\begin{proof}
As $C$ is a $J_{\zeta_\alpha}$-holomorphic curve of class $\alpha$, 
its area is $\int_{r(\alpha)} \omega_\zeta$. Using \eqref{eq_omega_zeta}, this can be rewritten as 
$\int_{r(\alpha)} \Rea(\zeta_\alpha^{-1}\Omega_I)= \Rea(\zeta_\alpha^{-1}\int_{r(\alpha)} \Omega_I)$. By \eqref{eq_zeta_alpha}, this is equal to $ \Rea(|\int_{r(\alpha)} \Omega_I|)
=|\int_{r(\alpha)} \Omega_I|$.
\end{proof}

\begin{lemma}
\label{lem_bound1}
Fix a norm $|\!|-|\!|$ on $H_2(M,L_1 \cup L_2,\ZZ) \otimes \RR$.
There exists a constant $A>0$ such that for every $p,p' \in L_1 \cap L_2$ and $\alpha \in \pi_1(p,p')$ with 
$C(p,p',\alpha) \neq 0$, we have 
\[ |\!|r(\alpha)|\!| \leq A \,|\int_{r(\alpha)} \Omega_I| \,.\]
\end{lemma}

\begin{proof}
Let $(h_i)$ be a basis of harmonic 2-forms on $M$ relative to $L_1 \cap L_2$ \cite{Hodge_nc}. It is enough to bound $|\int_{r(\alpha)} h_i|$ for every $i$. Let $A_i$ be the supremum of $|h_i|$ over all of $M$, which is finite if $M$ is bounded or appropriately convex at infinity. If $C(p,p',\alpha)\neq 0$, there exists a $J_{\zeta_\alpha}$-holomorphic curve $C$ of class $\alpha$, and so  
$|\int_{r(\alpha)} h_i|$
is bounded by $A_i$ times the area of $C$, which is equal to $|\int_{r(\alpha)} \Omega_I|$ by Lemma \ref{lem_area}.
\end{proof}

\begin{remark}
The inequality of Lemma \ref{lem_bound1} is formally similar to the support property in DT theory \cite{kontsevich2008stability} and its proof is parallel to the argument given in \cite[\S 1.2, Remark 1]{kontsevich2008stability} for the validity of the support property in the context of special Lagrangian submanifolds.
Under the correspondence between holomorphic Floer theory and DT theory described in \S\ref{sec_dt}, the inequality of Lemma \ref{lem_bound1} will exactly match with the support property in DT theory.
\end{remark}

\begin{lemma}
\label{lem_bound_2}
Fix $p, p' \in L_1 \cap L_2$ and $\beta \in H_2(M,L_1 \cup L_2,\ZZ)$.
Then, there exists finitely many $\alpha \in \pi_1(p,p')$ such that $C(p,p',\alpha) \neq 0$ and $r(\alpha)=\beta$. 
\end{lemma}

\begin{proof}
The complex structure $\zeta_\alpha$ and the area of $J_{\zeta_\alpha}$-holomorphic curves given by Lemma \ref{lem_area} only depend on $r(\alpha)$. Hence, we are considering a set of holomorphic curves with respect to a fixed complex structure and with fixed area, and so the result follows 
by Gromov's compactness for holomorphic curves.
\end{proof}

Let $(p,\gamma)$ and $(p',\gamma')$ be two critical points of $W$. We describe how to define a completed version $\hat{R}_{(p,\gamma),(p',\gamma')}$ of the Hom complex \eqref{eq_complex} used in the case of finitely many critical points.
Fix $N>0$.
For every $Q$ as in \eqref{eq_RQ}, corresponding to a sequence 
$(p,\gamma), (p_1,\gamma_1), \dots, (p_k,\gamma_k)$ of critical points, we say that 
$Q<N$ if 
\[\sum_j |\int_{r(\gamma_{j+1})-r(\gamma_j)} 
\Omega_I|<N\,.\] 
It follows from Lemmas \ref{lem_bound1}-\ref{lem_bound_2} 
that there are finitely many $Q$ such that $Q<N$ and 
\[R_Q := C((p,\gamma),(p_1,\gamma_1))\otimes C((p_1,\gamma_1),(p_2,\gamma_2)) \otimes \cdots \otimes C((p_{k},\gamma_k), (p',\gamma')) \neq 0\,.\]
Hence, the direct sum 
\[R_{(p,\gamma),(p',\gamma')}^{<N}:=
\bigoplus_{Q<N} R_Q \]
contains only finitely many non-zero summands and all the categorical constructions described in \S\ref{sec_FS_category} in the case of finitely many critical points make sense with $R_{(p,\gamma),(p',\gamma')}^{<N}$ and we obtain a truncated Fukaya-Seidel category $FS^{<N}(\widetilde{\cP}_{0}, W)$ with Hom-complexes $R_{(p,\gamma),(p',\gamma')}^{<N}$.
Finally, we define the Fukaya-Seidel category 
$FS(\widetilde{\cP}_{0}, W)$ by taking the categorical limit $C \rightarrow +\infty$, 
with Hom complexes
\[\hat{R}_{(p,\gamma),(p',\gamma')}=\lim_{N \rightarrow +\infty}R_{(p,\gamma),(p',\gamma')}^{<N}\,.\]

%By $\pi_1(\widetilde{\cP}_0)$-equivariance, we will have $\Hom((p,\gamma),(p',\gamma'))$

\subsection{2-categories from holomorphic symplectic manifold}
In the same way that Floer homology groups of pairs of Lagrangian submanifolds of a symplectic manifold can be realized as Hom spaces of the Fukaya category, it is natural to expect that the categories $FS(\widetilde{\mathcal{P}}_0,W)$
attached to pairs of holomorphic Lagrangian submanifolds of a holomorphic symplectic manifold can be realized as Hom categories of a 2-category.

\begin{conjecture}\label{conj_holom}
Let $M$ be a holomorphic symplectic manifold, with complex structure $I$ and holomorphic symplectic form $\Omega_I$.
The
$I$-holomorphic Lagrangian submanifolds of $M$
should form a linear $2$-category,  
such that the category of morphisms
 $\Hom (L_1,L_2)$ 
between two objects $L_1$ and $L_2$ 
is the Fukaya-Seidel category of the 
holomorphic action functional
on the infinite-dimensional space of paths between $L_1$ and $L_2$.
\end{conjecture}

We refer to the 2-category whose existence is asserted by Conjecture \ref{conj_holom} as the \emph{Fueter 2-category} of $M$, and we denote it by $\mathrm{Ft}(M)$, or $\mathrm{Ft}(M,\Omega_I)$ when we wish to make clear the complex structure and holomorphic symplectic form that are considered. We expect the Fueter 2-category to be independent on the choice of the auxiliary complex structures $J$ and $K$
(or to have only dependence on their asymptotic behaviour if $M$ is non-compact), in the same way that the Fukaya category is a symplectic invariant and independent on auxiliary choices of almost-complex structures. If $M$ is non-compact, one should also have a wrapped version of the Fueter 2-category, parallel to the wrapped Fukaya category \cite{wrapped}.

For the simplest examples of holomorphic symplectic manifolds given as cotangent bundles $T^{*}L$ of complex manifolds $L$, one expects the Hom category between the zero-section $L$ and the graph $L'$ of a holomorphic Morse function $W$ on $L$ to reproduce the Fukaya-Seidel category of $(L,W)$, similarly to the way that the Floer homology between the zero-section and a perturbation of the zero-section in the cotangent bundle of a manifold reproduces Morse homology
\cite{Flo2, MR1480992}. 

\begin{conjecture} \label{conj_dim_red}
Let $L$ 
be an exact Kähler manifold, $W$ 
a holomorphic Morse function on $L$, 
and $L'$ the graph of $dW$ in
$T^{*}L$. Then the Hom-category $\Hom_{\mathrm{Ft}(T^{*}L)} (L,L')$
in the Fueter 2-category of $T^{*}L$
coincides with the Fukaya-Seidel category of $(L,W)$.
\end{conjecture}

Conjectures \ref{conj_holom} and \ref{conj_dim_red} were independently formulated by Doan-Rezchikov
\cite{DR}. In \cite{DR}, Doan-Rezchikov introduce a flexible notion of hyperk\"ahler geometry and start the analytic study of the compactness properties of solutions of the 3d Fueter equation. As an application, they obtain  a first step towards Conjecture \ref{conj_dim_red} by proving an isomorphism between moduli spaces of solutions of the 3d Fueter equation in $T^{*}L$
and moduli spaces of pseudo-holomorphic curves in $L$.

\begin{remark}
A ``hyperkähler Floer theory" has been developed by Hohloch-Noetzel-Salamon
\cite{MR2555940}, which is also based on a version of the Fueter equation. However, the proposal
of \cite{MR2555940} and Conjecture \ref{conj_holom} are different: whereas the three complex structures $I$, $J$, $K$ play symmetric roles in \cite{MR2555940}, it is not the case in Conjecture \ref{conj_holom}.
\end{remark}

\subsection{The Fueter 2-category and the Rozansky-Witten 2-category}
\label{sec_rw}

As reviewed in \S\ref{sec:fd_complex_morse}, given a pair $(X,W)$ where $X$ is a K\"ahler manifold with $c_1(X)=0$ and $W\colon X \rightarrow \CC$ is a holomorphic Morse function, one can define the Fukaya-Seidel category $FS(X,W)$. It is a $\ZZ$-graded category which is not Calabi-Yau in general (the Serre functor is induced by the global monodromy around the critical points in the $W$-plane). Physically, $FS(X,W)$ is the category of boundary conditions in the A-model twist of the 2-dimensional 
$\mathcal{N}=(2,2)$
Landau-Ginzburg theory defined by $(X,W)$. Another category naturally attached to $(X,W)$ is the category of matrix factorizations $MF(X,W)$, which is $\ZZ/2\ZZ$-graded, Calabi-Yau, and physically obtained as the category of boundary conditions in the B-model twist of the Landau-Ginzburg theory defined by $(X,W)$.
When $(X,W)$ is the mirror of a Fano manifold $Y$, the Fukaya-Seidel category of $(X,W)$ is the derived category of coherent sheaves on $Y$:
\[ FS(X,W)=\mathrm{D^b Coh}(Y)\,,\]
and the category of matrix factorizations of $(X,W)$ is the Fukaya category of $Y$:
\[ MF(X,W)=F(Y)\,.\]

By construction, the Hom categories $\Hom_{\mathrm{Ft}(M,\Omega_I)}(L_1,L_2)$ in the
Fueter 2-category $\mathrm{Ft}(M,\Omega_I)$
of a holomorphic symplectic manifold $(M,\Omega_I)$ are the Fukaya-Seidel categories of the holomorphic action functional on the spaces of paths between $I$-holomorphic Lagrangian submanifolds $L_1$ and $L_2$.
If one replaces in this construction the Fukaya-Seidel category of the holomorphic action functional by the category of matrix factorizations of the holomorphic action functional, one obtains a 2-category $RW(M,\Omega_I)$.
The proposal for such a 2-category attached to a holomorphic symplectic manifold was first formulated by Kapustin-Rozansky-Saulina \cite{MR2522724, MR2771578}, who also argued that 
$RW(M,\Omega_I)$ should be the 2-category of boundary condition of the $3$-dimensional Rozansky-Witten topological quantum field theory (TQFT) \cite{RW}. In general, a 2-category which is Calabi-Yau, as $RW(M,\Omega_I)$, is always the category of boundary conditions of a 3d TQFT \cite{lurie}\footnote{We are neglecting here smoothness and compactness issues which could make the TQFT only partially defined.}. By contrast, a 2-category which is not Calabi-Yau, as the Fueter 2-category, defines at best a \emph{framed} 3d TQFT, which is only defined on framed manifolds of dimension $\leq 3$ \cite{lurie}.
This is a 3d version of the maybe more familiar fact that the B-model (resp.\ A-model) of a holomorphic Morse function $(X,W)$ (resp.\ of a Fano manifold $Y$) is a 2d TQFT whereas the A-model of $(X,W)$
(resp.\ the B-model of $Y$) is only a framed 2d TQFT.

One expects that most of the structures existing for the categories $FS(X,W)$ and $MF(X,W)$ of a finite-dimensional Landau-Ginzburg model $(X,W)$ have an analogue for the 2-categories $\mathrm{Ft}(M,\Omega_I)$ and $RW(M,\Omega_I)$ attached to a holomorphic symplectic manifold. For example, one can construct from $MF(X,W)$ a non-commutative Hodge structure \cite{KKP}, with Dolbeault data 
\[ \mathbb{H}^{*}(X, (\Omega_X^{*}, dW \wedge - ))\,,\]
de Rham data given by the $\CC(\!(u)\!)$-module
\[ \mathbb{H}^{*}(X, (\Omega_X^{*}, ud+dW \wedge-))\,\] 
and Betti data $H^{*}(X,X_{-\infty},\ZZ)$ where $X_{-\infty}=
\{ \Rea W <<0\}$. The Riemann-Hilbert type isomorphism between de Rham and complexified Betti data is given by the exponential periods
$\int_\gamma e^W \alpha$, where $\gamma \in H_{*}(X,X_{-\infty},\ZZ)$ are naturally classes of objects in $FS(X,W)$. Moreover, the Hodge filtration is encoded by a connection on the Rham data, with second order pole at $u=0$, and space of flat sections given by the Betti data. Finally, the Stokes data of the irregular singularity of the connection at $u=0$, controlling the jump of the asymptotic expansion of the periods $\int_\gamma e^{W/u} \alpha$ obtained by expressing $\gamma$ as a linear combination of Lefschetz thimbles for $\Rea(W/u)$,
is determined by the counts of 2d BPS states of $(X,W)$ (which are equal to Euler characteristics of the spaces of 2d BPS states reviewed in \S\ref{sec_2d_bps} for the construction of $FS(X,W)$). 

It is conjectured in \cite{KKP} that there is more generally an entirely categorical way to produce a non-commutative Hodge structure from a smooth proper dg-category $C$. For example, the Dolbeault data should be the Hochschild homology $HH(C)$ and the de Rham data should be the periodic cyclic homology $HC(C)$. We propose that this construction should have a 2-categorical analogue. For example, if $C$ is a Calabi-Yau 2-category, as $RW(M,\Omega_I)$, one can define its Hochschild homology $HH(C)$ as the category obtained by evaluating the corresponding 3d TQFT on the circle $S^1$, and similarly
its periodic cyclic homology $HC(C)$ as the category over $\CC(\!(u)\!)$ obtained by $S^1$-equivariant compactification of the 3d TQFT over $S^1$.

\begin{conjecture}
\label{conj_rw}
Let $(M,\Omega_I)$ be a holomorphic symplectic manifold. Then,
\begin{itemize}
    \item[(i)] 
the Hochschild homology of the 2-category $RW(M,\Omega_I)$ is the derived category of coherent sheaves on $(M,I)$:
\[ HH(RW(M,\Omega_I))=\mathrm{D^b Coh}(M,I)
\,.\]
\item[(ii)] the periodic cyclic homology of the 2-category $RW(M,\Omega_I)$ is the category of DQ-modules \cite{MR1855264, KS_DQ}
obtained by deformation quantization of $(M,\Omega_I)$:
\[ HP(RW(M,\Omega_I))=DQ(M,\Omega_I)\,.\]
\end{itemize}
\end{conjecture}

Conjecture \ref{conj_rw}(i) is the known expectation that the Rozansky-Witten theory compactified on $S^1$ is the B-model \cite{MR2522724, MR2771578}, whereas the formulation of Conjecture \ref{conj_rw}(ii) might be new.

Finally, we describe how holomorphic Floer theory should be related to the Betti data of the hypothetical categorical non-commutative Hodge structure attached to $RW(\Omega_I,W)$.
For every $u \in \CC^{*}$, $\omega_u=\Rea(u^{-1}\Omega_I)$ is a symplectic form on $M$, and so one can consider
the Fukaya category of the symplectic manifold $(M,\omega_u)$. The Fukaya category is generally defined over a Novikov ring whose variable keeps track of the area of $J_u$-holomorphic curves.
From now on, assume that for $u>0$ small enough, the Fukaya category converges when we replace the Novikov parameter $q^\beta$ by
\[ e^{-\int_{\beta}(\omega_u+iB_u)}=e^{-\frac{1}{u}\int_\beta \Omega_I}\,,\]
where $B_u:=\Ima(u^{-1}\Omega_I)$, and we denote by $F(M,\omega_u, B_u)$ the corresponding category. Viewed as a family over the formal punctured disk $\CC(\!(u)\!)$, we obtain a category $F(M,\omega_u, B_u)\otimes \CC(\!(u)\!)$ over $\CC(\!(u)\!)$ where the contributions of $J_u$-holomorphic curves disappear because proportional to $e^{-\frac{1}{u}\int_\beta \Omega_I}$, which has zero Taylor expansion in $u$. Note that $J_u$-holomorphic disks only exist when the argument of $u$ is in the complement of the countable set of arguments of complex numbers $\int_\beta \Omega_I$ indexed by relative homology classes $\beta \in H_2(M,L_1\cup L_2,\ZZ)$ between 
holomorphic Lagrangian submanifolds $L_1$ and $L_2$.
The following conjecture is motivated by the known finite-dimensional story for $FS(X,W)$ and $MF(X,W)$.

\begin{conjecture} \label{conj_dq}
The Betti data for the non-commutative Hodge theory of $RW(M,\Omega_I)$ is induced by a Riemann-Hilbert type isomorphism 
\begin{equation} \label{eq_DQ_HFT}
DQ(M,\Omega_I)\simeq F(M,\omega_u, B_u)\otimes \CC(\!(u)\!)\,. \end{equation}
Moreover, the Stokes data is induced by the holomorphic dependence of $F(M,\omega_u, B_u)$ in $u$ and is determined by the counts of $J_u$-holomorphic curves entering the definition of the spaces of 2d BPS states of the holomorphic action functional described in \S\ref{sec_2d_bps_HFT} for the construction of $\mathrm{Ft}(M,\Omega_I)$.
\end{conjecture}

A relation between deformation quantization and the Fukaya category is predicted from physics arguments by Kapustin in \cite{kapustin2005branes}, and related works include \cite{BF_DQ,J_DQ,M_DQ, P_DQ, SV}.
Moreover, such a relation combined with the Stokes data interpretation of $J_u$-holomorphic curves is one of the main points of the work of Kontsevich-Soibelman \cite{KS_HFT} on non-perturbative quantization of $(M,\Omega_I)$ and the topic of resurgence. 
Conjecture \ref{conj_dq} suggests a formulation of this relation in terms of 
non-commutative Hodge theory for $RW(M,\Omega_I)$.

\subsection{Complex and holomorphic Atiyah-Floer conjectures}

Holomorphic Floer theory is an example of infinite-dimensional complex Morse theory. In this section, we describe conjectural relations between holomorphic Floer theory and the two other examples of infinite-dimensional Morse theory briefly reviewed in \S\ref{sec_inf_ex} and given by complex-valued Chern-Simons theory and holomorphic Chern-Simons theory. These conjectures are respectively complex and holomorphic analogues of the Atiyah-Floer conjecture relating the two examples of infinite-dimensional Morse theory given by Lagrangian Floer theory and Chern-Simons theory (Floer homology for 3-manifolds) \cite{AF}.

Let $X$ be a compact 3-manifold with a choice of Heegaard splitting $X=X_1 \cup_\Sigma X_2$ along a compact surface $\Sigma$. Let $G_\CC$ be the complexification of a compact Lie group. The character variety $M_{\Sigma,G_\CC}$ parametrizing flat $G_\CC$-connections on $\Sigma$ is naturally a holomorphic symplectic manifold, and the locus $L_{X_1}$ (resp.\ $L_{X_2}$) of $G_\CC$
flat connections on $\Sigma$ extending to $X_1$ (resp.\ $X_2$) is a holomorphic Lagrangian submanifold of $M_{\Sigma,G_\CC}$.
Holomorphic Floer theory applied to the pair of holomorphic Lagrangian submanifolds $L_1$ and $L_2$ conjecturally produces a
Hom category $\Hom_{\mathrm{Ft}(M_{\Sigma,G_\CC})}(L_{X_1},L_{X_2})$
in the Fueter 2-category of $M_{\Sigma,G_\CC}$. On the other hand, complex-valued Chern-Simons theory conjecturally attaches to the 3-manifold $X$ a category $CS(X)$ constructed from counts of flat $G_\CC$-connections on $X$, and of solutions to the Kapustin-Witten and Haydys-Witten equations on $X\times \RR$ and $X\times \RR^2$.

\begin{conjecture}\label{conjAF1}
Let $X$ be a compact 3-manifold with a choice of Heegaard splitting $X=X_1 \cup_\Sigma X_2$ along a compact surface $\Sigma$, and let $G_\CC$ be the complexification of a compact Lie group. Then, there exists an equivalence of categories
\[ \Hom_{\mathrm{Ft}(M_{\Sigma,G_\CC})}(L_{X_1},L_{X_2}) \simeq CS(X) \,.\]
\end{conjecture}

Let $X$ be a compact Calabi-Yau 3-fold degenerating to the union $X_1 \cup_\Sigma X_2$ of two Fano 3-folds transversally glued along 
a common anticanonical $K3$ surface $\Sigma$. The moduli space $M_{\Sigma,\gamma}$ of stable holomorphic vector bundles on $\Sigma$ with Chern classes $\gamma$ is naturally a holomorphic symplectic manifold, and the locus $L_{X_1}$ (resp.\ $L_{X_2}$) of holomorphic vector bundles on $\Sigma$ extending to $X_1$ (resp.\ $X_2$) is a holomorphic Lagrangian submanifold of $M_{\Sigma,\gamma}$.
Holomorphic Floer theory applied to the pair of holomorphic Lagrangian submanifolds $L_1$ and $L_2$ conjecturally produces a
Hom category $\Hom_{\mathrm{Ft}(M_{\Sigma,\gamma})}(L_{X_1},L_{X_2})$
in the Fueter 2-category of $M_{\Sigma,\gamma}$. On the other hand, holomorphic Chern-Simons theory conjecturally attaches to the Calabi-Yau 3-fold $X$ and to the Chern classes $\gamma$ a category $HCS(X,\gamma)$ constructed from counts of stable holomorphic vector bundles on $X$, and of $G_2$-instantons and $Spin(7)$-instantons on $X\times \RR$ and $X\times \RR^2$.

\begin{conjecture}\footnote{Conjecture \ref{conjAF2} was suggested to the author by Richard Thomas. The uncategorified statement relating counts of holomorphic vector bundles and Lagrangian intersection numbers is discussed in \cite[\S 4]{DTgauge}.}
\label{conjAF2}
Let $X$ be a compact Calabi-Yau 3-fold degenerating to the union $X_1 \cup_\Sigma X_2$ of two Fano 3-folds transversally glued along 
a common anticanonical $K3$ surface $\Sigma$.
Then, there exists an equivalence of categories \[ \Hom_{\mathrm{Ft}(M_{\Sigma,\gamma})}(L_{X_1},L_{X_2}) \simeq HCS(X,\gamma) \,.\]
\end{conjecture}

\begin{remark}
As for the usual Atiyah-Floer conjecture, one expects Conjectures \ref{conjAF1} and \ref{conjAF2} to follow from an adiabatic dimensional reduction of the higher dimensional gauge theoretic equations.
\end{remark}

\section{2d $\mathcal{N}=(2,2)$ theories and  holomorphic Floer theory}
\label{sec_2d}

In \S\ref{sec:details_HFT}, we introduced holomorphic Floer theory as a Landau-Ginzburg model with an infinite dimensional target space. In this section, we explain 
that holomorphic Floer theory is actually relevant to the study of usual 
2d $\mathcal{N}=(2,2)$ field theories such as Landau-Ginzburg models with finite dimensional target spaces. The conjectures that we formulate here will motivate our conjectures on the relation between holomorphic Floer theory and DT invariants in \S\ref{sec_dt} (see Remark \ref{rem_2d_4d}).

\subsection{General 2d $\mathcal{N}=(2,2)$ theories}

Let $\mathcal{T}$ be a massive 2-dimensional 
$\mathcal{N}=(2,2)$
field theory with $n$ vacua. We denote by $\mathrm{Br}(\mathcal{T})$ its $A_\infty$-category of boundary conditions (or branes) as considered in \cite{GMWinfrared}. When $\mathcal{T}$ is the Landau-Ginzburg model 
defined by a holomorphic function $W \colon X \rightarrow \CC$
on a Kähler manifold $X$, then $\mathrm{Br}(\mathcal{T})$ is the Fukaya-Seidel category of $(X,W)$. The goal of this section is to argue that for general $\mathcal{T}$, the category of boundary conditions
$\mathrm{Br}(\mathcal{T})$ can be extracted from a Fueter 2-category constructed from the moduli space of deformations of 
$\mathcal{T}$. 

Let $\mathcal{N}_{\mathcal{T}}$ be the generically semi-simple Frobenius manifold parametrizing  deformations of 
$\mathcal{T}$ \cite{CV, Dub}. 
In particular, the tangent bundle $T\mathcal{N}_{\mathcal{T}}$
has a structure of sheaf of commutative algebras over $\mathcal{N}_{\mathcal{T}}$, coming from the identification of the deformations of $\mathcal{T}$ with the the chiral ring. The natural quotient map 
\[ \mathrm{Sym}(T \mathcal{N}_{\mathcal{T}}) \rightarrow  
T \mathcal{N}_{\mathcal{T}} \]
induces an inclusion 
\[ \mathcal{L}=\mathrm{Spec}(T \mathcal{N}_{\mathcal{T}}) \hooklongrightarrow 
T^{*} \mathcal{N}_{\mathcal{T}} = \mathrm{\Spec}(\mathrm{Sym}(T \mathcal{N}_{\mathcal{T}})) \,.\]
The assumption that $\mathcal{T}$ is massive with $n$ vacua implies that the spectral cover 
$\mathcal{L} \rightarrow \mathcal{N}_{\mathcal{T}}$ is generically of degree $n$. Furthermore, the inclusion 
$\mathcal{L} \subset T^{*} \mathcal{N}_{\mathcal{T}}$
naturally realizes $\mathcal{L}$ as a holomorphic Lagrangian submanifold of the holomorphic symplectic variety $T^{*} \mathcal{N}_{\mathcal{T}}$
by \cite[Corollary 1.9]{Aud}. On the other hand, let $F_{\mathcal{T}} \subset T^{*} \mathcal{N}_{\mathcal{T}}$ be the holomorphic Lagrangian 
submanifold given by the cotangent fiber over the point of 
$\mathcal{N}_{\mathcal{T}}$ corresponding to $\mathcal{T}$.
The intersection $F_{\mathcal{T}} \cap \mathcal{L}$
consists of $n$ points in natural correspondence with the 
$n$ vacua of $\mathcal{T}$, see Figure \ref{fig4}.

\begin{figure}[ht!]
\centering
\includegraphics[width=50mm]{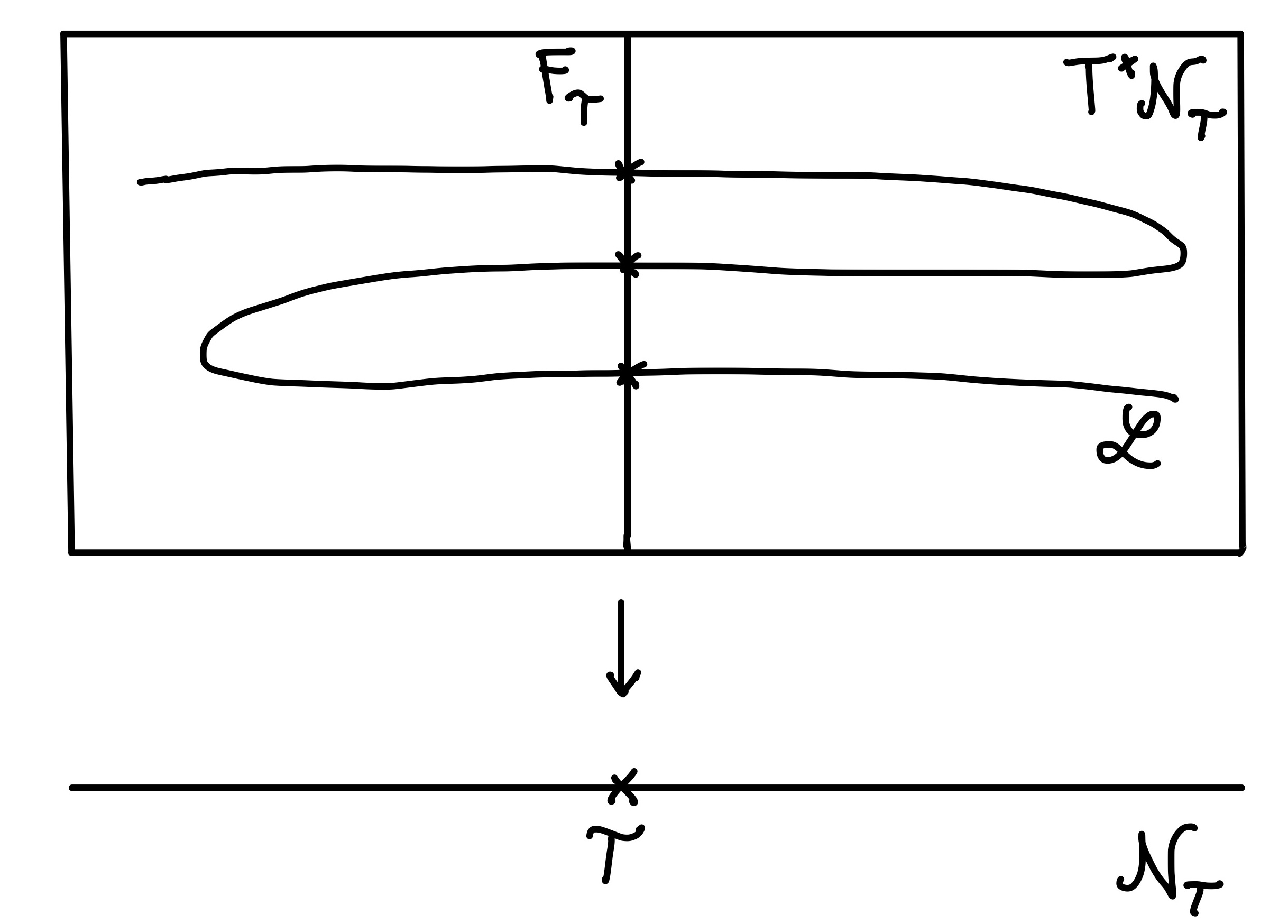}
\caption{The spectral cover $\mathcal{L}$ and the cotangent fiber $F_{\mathcal{T}}$. \label{fig4}}
\end{figure}

\begin{conjecture}\label{conj_general_2d} The $A_\infty$-category of boundary conditions $\mathrm{Br}(\mathcal{T})$ of a massive 
2-dimensional $\mathcal{N}=(2,2)$ field theory $\mathcal{T}$ is equivalent to the Hom-category in the Fueter 2-category 
$\mathrm{Ft}(T^{*}\mathcal{N}_{\mathcal{T}})$ of 
between the spectral cover $\mathcal{L}$ and the cotangent fiber $F_{\mathcal{T}}$: 
\begin{equation}\label{eq_conj_general_2d}
\mathrm{Br}(\mathcal{T}) \simeq \mathrm{Hom}_{\mathrm{Ft}(T^{*}\mathcal{N}_{\mathcal{T}})}(\mathcal{L}, F_{\mathcal{T}}) \,.\end{equation}
\end{conjecture}

\begin{remark}
Moving the point $\mathcal{T}$ in $\mathcal{N}_{\mathcal{T}}$, that is moving the cotangent fiber $F_{\mathcal{T}}$, the $A_\infty$-categories 
in \eqref{eq_conj_general_2d} should form a ``perverse sheaf of categories" over $\mathcal{N}_{\mathcal{T}}$, and so fit in the theory of perverse schobers \cite{kapranov2014perverse, kapranov2020perverse}.
\end{remark}

\subsection{Surface defects and spectral networks}

We also formulate below a version of Conjecture \ref{conj_general_2d}
in the context of theories of class $\mathcal{S}$.
Let $C$ be a Riemann surface (possibly with punctures), and
$B$ the base of the Hitchin fibration on the moduli space of $SL_n(\mathbb{C})$ Higgs bundles on $C$. For every point 
$b \in B$, we have the corresponding spectral curve 
$\Sigma_b \subset T^{*}C$. When $\Sigma_b$ is smooth, the projection 
$\Sigma_b \rightarrow C$ is a ramified cover of degree $n$. 

In physics language, $\Sigma_b$ parametrizes  
2-dimensional $\mathcal{N}=(2,2)$ field theories describing the ``canonical surface defect" 
probing the vacuum $b \in B$ of the 
4-dimensional $\mathcal{N}=2$ of class $\mathcal{S}$ obtained by compactifying the 
6-dimensional $\mathcal{N}=(2,0)$ superconformal field theory of type $A_{n-1}$ on $C$ \cite[\S 7]{GMN_2d_4d}. For every $x \in \Sigma_b$, we denote by $\mathcal{T}_{x,b}$ the corresponding 
2-dimensional $\mathcal{N}=(2,2)$ field theory and by $F_x$ the cotangent fiber of $T^* C$ over $x$.

\begin{conjecture} \label{conj_general_2d_4d} 
For every $b \in B$ and $x \in C$, the $A_\infty$-category of boundary conditions
$\mathrm{Br}(\mathcal{T}_{x,b})$ of the canonical surface defect 2-dimensional $\mathcal{N}=(2,2)$  field theory $\mathcal{T}_{x,b}$ is 
equivalent to the Hom-category in the Fueter 2-category $\mathrm{Ft}(T^* C)$ of $T^{*}C$ between the spectral curve 
$\Sigma_b$ and the cotangent fiber $F_x$:
\begin{equation} \mathrm{Br}(\mathcal{T}_{x,b}) \simeq \mathrm{Hom}_{\mathrm{Ft}(T^*C)}
(\Sigma_b,F_x) \,.\end{equation}
\end{conjecture}

\begin{remark}
Conjectures \ref{conj_general_2d} and \ref{conj_general_2d_4d}
are of a similar nature. The only difference is that in Conjecture 
\ref{conj_general_2d_4d}, $C$ is a space of particular deformations of 
$\mathcal{T}_{x,b}$, whereas in Conjecture \ref{conj_general_2d}, 
$\mathcal{N}_{\mathcal{T}}$ is the universal space of deformations of 
$\mathcal{T}$.
\end{remark}

A non-trivial evidence for Conjecture \ref{conj_general_2d_4d}
is provided by the M-theory realization of the BPS states of 
$\mathcal{T}_{x,b}$ and their description by spectral networks.
Indeed, the M-theory realization of the BPS states of 
$\mathcal{T}_{x,b}$ is by BPS M2-branes with boundary on 
$F_x \cup \Sigma_b$, projecting onto spectral networks on $C$ ending 
at $x$ \cite[\S 7]{GMN_2d_4d}. BPS M2-branes are holomorphic disks with respect to a hyperkähler rotated complex structure prescribed by the phase of the central charge, and so are exactly the objects involved in the first step of the construction of  $\mathrm{Hom}_{\mathrm{Ft}(T^*C)}
(F_x,\Sigma_b)$. The second step involving solutions of the Fueter equations seems to be unexplored. It is an interesting question to find out what is the corresponding generalization in the description by spectral networks.

\section{DT invariants from holomorphic Floer theory}
\label{sec_dt}

After an overview of the geometry of Seiberg-Witten integrable systems in \S \ref{sec_sw}, we formulate in \S\ref{sec_main_hft_conj} our main conjecture relating holomorphic Floer theory and DT invariants occuring as BPS invariants of 4-dimensional $\cN=2$ 
field theories.
We present a heuristic derivation of this conjecture in \S\ref{sec_derivation}. Finally, we formulate further conjectures involving categories of line operators in \S\ref{sec_lines}.

\subsection{Seiberg-Witten integrable systems}
\label{sec_sw}

The low energy physics of a rank $r$ 4-dimensional
$\mathcal{N}=2$ 
field theory (without dynamical gravity) $\mathcal{T}$ is controlled by its Seiberg-Witten complex integrable system 
\[ \pi \colon M \longrightarrow B\,,\] 
where $B$ is the Coulomb branch of $\mathcal{T}$ on $\RR^4$ and $M$ is the Coulomb branch of 
$\mathcal{T}$ on $\RR^3 \times S^1$.
While $B$ is a complex manifold of dimension $r$, the space $M$
is a hyperk\"ahler manifold of complex dimension $2r$ \cite{SW_3d}. The low-energy effective theory of $\mathcal{T}$ on $\RR^3 \times S^1$ is the 
3d $\mathcal{N}=4$ sigma-model with target $M$.

We denote by $I$ the complex structure on $M$ in which $\pi$ is holomorphic and by $\Omega_I$ the corresponding $I$-holomorphic symplectic form.
As in \S\ref{sec_hsg}, we denote by $J$ and $K$ the complex structures such that $\Omega_I=\omega_J+i\omega_K$, and for every $\zeta \in \CC$, $|\zeta|=1$, we set 
$J_\zeta := (\Rea\, \zeta) J + (\Ima\, \zeta) K$.
The fibers $F_b:= \pi^{-1}(b)$ over points $b \in B$ in the complement of the discriminant locus $\Delta$ of $\pi$
are $r$-dimensional abelian varieties and $I$-holomorphic Lagrangian submanifolds of $(M,\Omega_I)$, see Figure \ref{fig6}.
The lattice of charges for states in the vacuum $b \in B \setminus \Delta$ is 
\[ \Gamma_b := \pi_2(M,F_b) \]
and for every $\gamma \in \Gamma_b$, we have a space 
$BPS_\gamma^b$ of BPS states of charge $\gamma$
in the vacuum $b$
\cite{MR1293681, MR1306869, SW_3d, GMN1, GMN2}.

\begin{figure}[ht!]
\centering
\includegraphics[width=40mm]{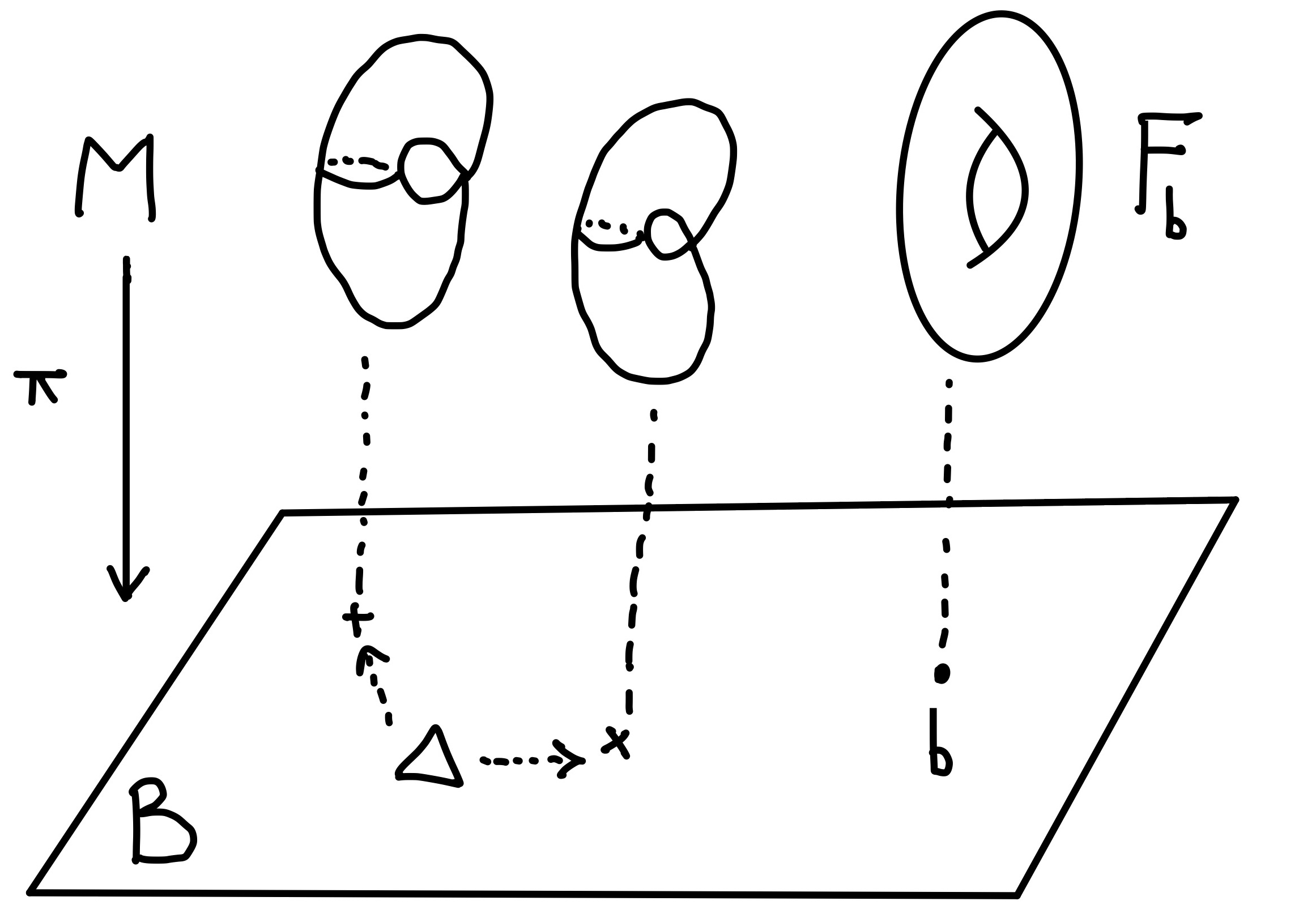}
\caption{A rank $1$ Seiberg-Witten integrable system $\pi: M \rightarrow B$, with discriminant locus $\Delta$. \label{fig6}}
\end{figure}

In the vacuum $b \in B \setminus \Delta$, the low-energy effective theory is a $U(1)^r$ gauge theory with electromagnetic charge lattice $H_1(F_b,\ZZ)$. 
The natural map $\Gamma_b \rightarrow H_1(F_b,\ZZ)$ is the projection of the lattice of charges on the lattice of electromagnetic charges and its kernel is the lattice of flavour charges.
As already mentioned in \S\ref{sec_intro_dt},
in many cases, there is an associated 3-dimensional Calabi-Yau (CY3) triangulated category $\mathcal{C}$, the base 
$B \setminus \Delta$ maps to the space of Bridgeland stability conditions on $\mathcal{C}$, the central charge at a point $b\in B\setminus \Delta$ is given by 
\[ Z_\gamma(b)=\int_\gamma \Omega_I\,, \]
and $BPS_\gamma^b$
is mathematically realized as a cohomological DT invariant of $\mathcal{C}$ counting $b$-stable objects of class $\gamma$. 

\begin{example}
When $\mathcal{T}$ is a class $\mathcal{S}$ theory, obtained by compactifying the 
6-dimensional $\mathcal{N}=(2,0)$ 
superconformal field theory of ADE type $G$
on a Riemann surface $C$ (possibly with punctures), the corresponding Seiberg-Witten complex integrable system
is closely related to the Hitchin integrable system on the moduli space of $G$-Higgs bundles on $C$. Typically, the Seiberg-Witten integrable system is a self-dual fiberwise finite quotient of an Hitchin integrable system \cite{Derry, GMN1, GMN2}. In particular, the base $B$ is isomorphic to $\CC^r$ as a complex manifold.   
\end{example}

\begin{example}
By considering a 4-dimensional $\mathcal{N}=2$ 
theory obtained by an appropriate compactification of a six-dimensional little string theory on $T^2$, one can obtain an elliptic K3 surface as Seiberg-Witten integrable system \cite{K3_coulomb}. In this case, the base $B$ is a complex projective line. 
\end{example}

From now on, we assume that $B$ is simply-connected. This assumption covers in particular the cases of Hitchin systems and K3 surfaces that we just described.

\subsection{DT invariants and holomorphic Floer theory}
\label{sec_main_hft_conj}
Let $D$ be a cigar geometry: $D$ is the closed unit disk
with polar coordinates $r\in [0,1]$, $\theta \in \RR/2\pi\ZZ$ and a metric of the form $ds^2=dr^2+f(r)d\theta^2$ for a function $f(r)$ such that there exists $0<\epsilon <<1$ such that $f(r)\sim r^2$ for $r<<\epsilon$ and $f(r)$ is equal to a constant $\rho$ for $\epsilon \leq r \leq 1$. The map $r: D \rightarrow [0,1]$ is a $S^1$-bundle for $r \neq 0$ and the $S^1$-fiber shrinks to a point at $r=0$, see Figure \ref{fig5}.

\begin{figure}[ht!]
\centering
\includegraphics[width=40mm]{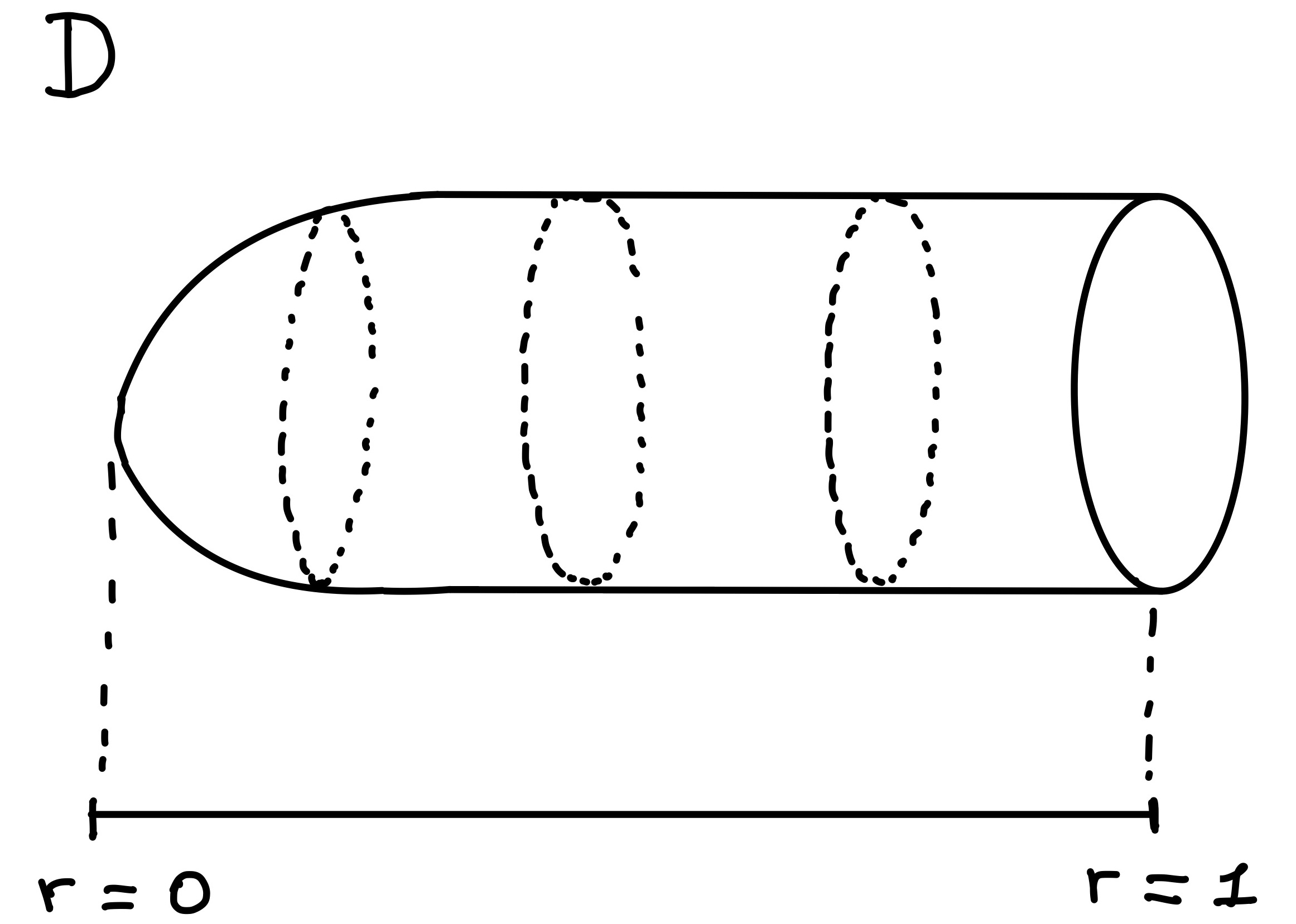}
\caption{The cigar $D$. \label{fig5}}
\end{figure}

We consider the 4d $\mathcal{N}=2$ theory $\mathcal{T}$ on $\RR^2 \times 
D$ with a topological twist along $D$ preserving half of the supersymmetries \cite[\S 3.1]{NW}.
In the limit where the radius of the circle $S^1$ in $D$ shrinks to zero, the low energy description is the 3d sigma model on $\RR^2 \times [0,1]$
with target $M$ and a boundary condition at $r=0$ defined by the tip of the cigar. It is argued by Nekrasov-Witten \cite[\S 3.1]{NW} that this boundary condition is defined by a $I$-holomorphic Lagrangian section $S$ of $\pi \colon M \rightarrow B$. For example, for $b \in B \setminus \Delta$, the fiber $F_b=\pi^{-1}(b)$ is parametrized by a complexification of the holonomy around $S^1$ of the 4d low energy $U(1)^r$ gauge field, and the shrinking of $S^1$ at the tip of the cigar forces this holonomy to vanish. When $\mathcal{T}$ is a class S theory, $M$ is a moduli space of Higgs bundles and $S$ is the Hitchin section of the Hitchin fibration \cite[\S 4.6]{NW}.

In Conjecture \ref{conj_main} below, we propose that the BPS spectrum of $\mathcal{T}$ at a point $b \in B \setminus \Delta$ can be recovered from holomorphic Floer theory of $(M,\Omega_I)$ for the pair of $I$-holomorphic Lagrangian submanifolds $S$
and $F_b$. To formulate this proposal, we first need some preliminary results on the space of paths $\mathcal{P}$ between $S$ ad $F_b$.
As $S$ is a section of $\pi$ and $F_b$ is a fiber, the intersection 
$S \cap F_b$ consists of a single point $p$, see Figure \ref{fig7}:
\[ S \cap F_b=\{p\}\,.\]
As in \S\ref{sec_hol_action}, we denote by $\cP_0$ the connected component of $\cP$ containing the constant path $p$.

\begin{figure}[ht!]
\centering
\includegraphics[width=40mm]{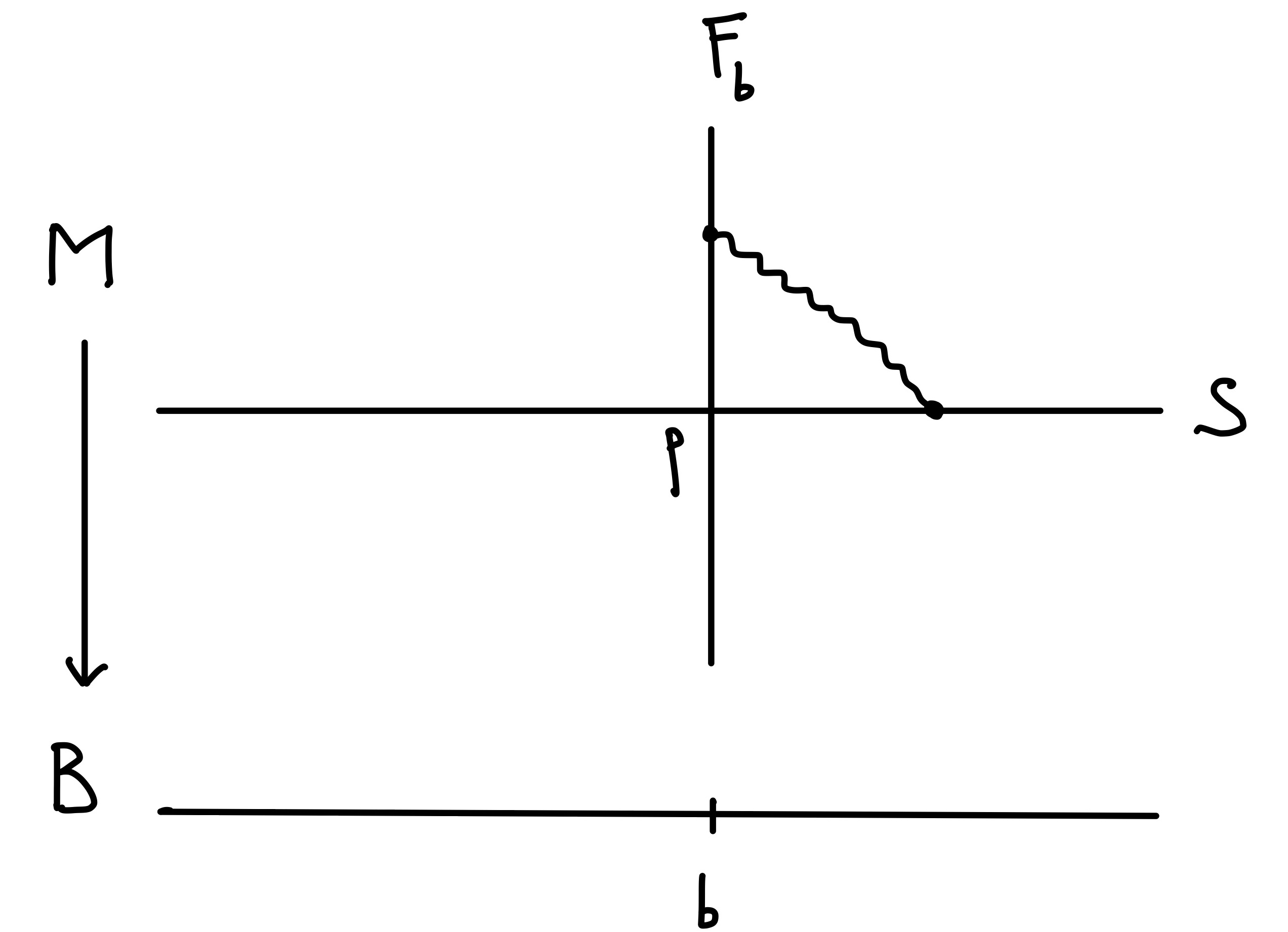}
\caption{A path with endpoints on the section $S$ and on the fiber $F_b$. \label{fig7}}
\end{figure}

\begin{lemma}
\label{lem_pi_1}
The fundamental group 
$\pi_1(\cP_0)$ is 
naturally isomorphic to the charge lattice 
$\Gamma_b=\pi_2(M, F_b)$.
\end{lemma}

\begin{proof}
Let $w$ be a loop in $\cP_0$ based at $p$:
\begin{align*} w: &[0,1] \longrightarrow \cP_0 \\
 x &\mapsto (w_x : y \in [0,1] \mapsto w_x(y))\,,
\end{align*}
and $w_0(y)=w_0(y)$ for all $y\in [0,1]$, $w_x(0) \in S$ and $w_x(1) \in F_b$ for all $x \in [0,1]$.
As $S \simeq B$ is assumed to be simply-connected, one can continuously deform $w$ to contract the loop 
$x \mapsto w_x(0)$ in $S$, and so assume from now on
without loss of generality that $w_x(0)=p$ for all $x \in 
[0,1]$. 

The map 
$y \mapsto (x \mapsto w_x(y))$
is a one parameter family of loops
in $M$, starting at $y=0$ with the constant loop at $p$, and ending with the loop $x \mapsto w_x(1)$ in $F_b$. In particular, we obtain a map from a disk 
to $M$ with boundary mapping to $F_b$, see Figure \ref{fig8}.
The loop $w$ is homotopically trivial in $\cP_0$ if and only if this disk is homotopically trivial in 
$M$ relatively to $F_b$.
\end{proof}

\begin{figure}[ht!]
\centering
\includegraphics[width=40mm]{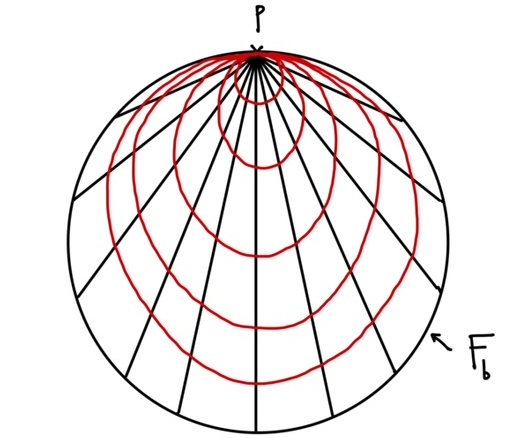}
\caption{The one-parameter family of paths $x \mapsto (y \mapsto w_x(y))$, and the one-parameter family of loops $y \mapsto (x \mapsto w_x(y))$, forming a disk with boundary on $F_b$. \label{fig8}}
\end{figure}

Let $\widetilde{\cP}_0 \rightarrow \cP_0$ be the universal cover of $\cP_0$. As the intersection $S \cap F_b$ consists of a single point $p$, the critical points of the holomorphic action functional on $\widetilde{\cP}_0$ form a $\pi_1(\cP_0)$-torsor by \S\ref{sec_inf_dim}. By Lemma \ref{lem_pi_1}, $\pi_1(\cP_0)=\Gamma_b$, and so we have a $\Gamma_b$-torsor of critical points.  
Fixing a reference critical point, we identify the set of critical points with $\Gamma_b$, and we denote by $(p,\gamma)$ the critical point corresponding to $\gamma \in \Gamma_b$. By \S \ref{sec_2d_bps_HFT}, holomorphic Floer theory attaches to any pair of critical points $(p,\gamma)$ and $(p,\gamma')$ a vector space $H((p,\gamma),(p,\gamma'))$ of 2d BPS states, see Figure \ref{fig9}.
By $\Gamma_b$-equivariance of the holomorphic action functional on $\mathcal{P}_0$,
the space $H((p,\gamma),(p,\gamma'))$ only depend on the difference $\gamma'-\gamma$. By definition, $H((p,0),(p,\gamma))$ is the cohomology of a complex generated by $J_\zeta$-holomorphic curves in $M$ with boundary on $S \cup F_b$, and with differential given by counts of solutions to the 3d $\zeta$-Fueter equation asymptotic to these holomorphic curves, where
\[\zeta=\frac{\int_\gamma \Omega_I}{|\int_\gamma \Omega_I|}\,. \]

\begin{figure}[ht!]
\centering
\includegraphics[width=40mm]{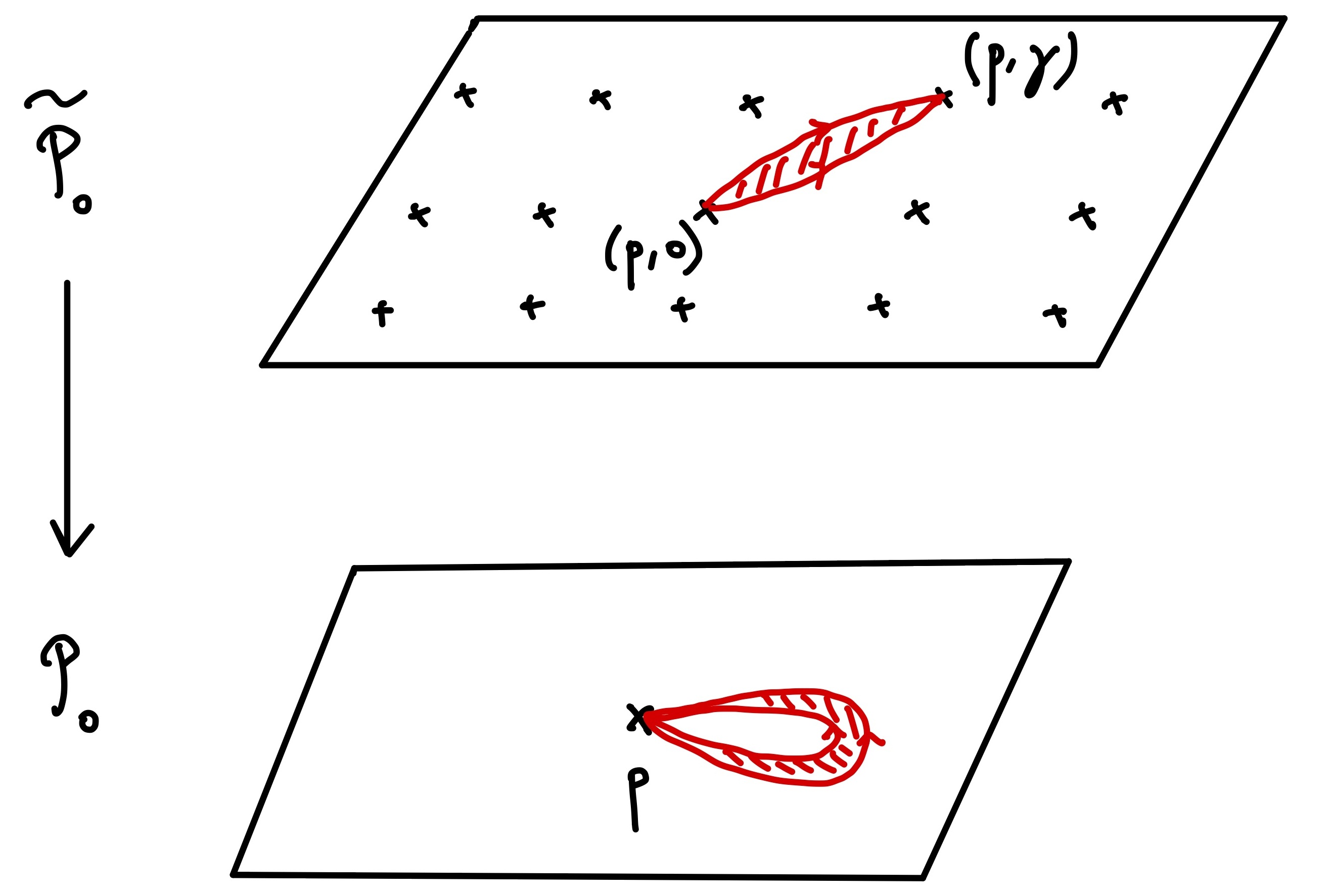}
\caption{The lattice of critical points of $W$ on $\widetilde{\cP}_0$.\label{fig9}}
\end{figure}

We can now state our main conjecture relating the BPS states of the 4d $\mathcal{N}=2$
theory $\mathcal{T}$ and the 2d BPS states of the holomorphic action functional.

\begin{conjecture} \label{conj_main}
Let $\mathcal{T}$ be a 4d-dimensional $\mathcal{N}=2$ field theory.
For every point $b \in B \setminus \Delta$ in the complement of the discriminant locus $\Delta$ of the Coulomb branch $B$ of $\mathcal{T}$, and for every charge $\gamma \in \Gamma_b$, the space $BPS_\gamma^b$
of BPS states of $\mathcal{T}$ of charge $\gamma$ in the vacuum $b$ is isomorphic to the space
$H((p,0),(p,\gamma))$
of 2d BPS states of the holomorphic action functional
on the space of paths between the holomorphic Lagrangian section $S$ and the holomorphic Lagrangian fiber $F_b$ of the Seiberg-Witten integrable system $(M,\Omega_I)$ of $\mathcal{T}$:
\[ BPS_\gamma^b \simeq H((p,0),(p,\gamma)) \,.\]
\end{conjecture}

When the 4d $\mathcal{N}=2$ theory $\mathcal{T}$ can be geometrically engineered by compactification of Type II string theory on a Calabi-Yau 3-fold, then the spaces of BPS states $BPS_\gamma^b$
have a geometric realization as DT invariants, and so in this case Conjecture \ref{conj_main} predicts a relation between the DT theory of a Calabi-Yau 3-fold and the holomorphic Floer theory of the complex integrable system $\pi \colon M \rightarrow B$.

\begin{remark}
In the decategorified limit, Conjecture \ref{conj_main}
predicts a relation between BPS indices, or DT invariants, and counts of $J_\zeta$-holomorphic curves in $M$ with boundary on $S \cup F_b$. In fact, it follows from the proof of Lemma \ref{lem_pi_1} that these $J_\zeta$-holomorphic curves can be naturally viewed as $J_\zeta$-holomorphic disks in $M$ with boundary on the fiber $F_b$. 

The counts of these disks are exactly the ``quantum corrections" appearing in mirror symmetry: for every $\zeta \in \CC$, $|\zeta|=1$, the map $\pi \colon M \rightarrow B$ is a torus fibration in special Lagrangian on $(M,\omega_\zeta, J_\zeta)$, and so the Strominger--Yau--Zaslow picture
of mirror symmetry \cite{syz} predicts that the mirror of $(M,\omega_\zeta, J_\zeta)$ is obtained by dualizing
the torus fibration away from the singular fibers, and using counts of $J_\zeta$-holomorphic disks with boundary on the torus fibers $F_b$ to extend the complex structure of the mirror across the singular fibers \cite{Fuk, GSaffine, KSaffine}. 
For a general Calabi-Yau manifold $M$, the wall-crossing behaviour of these counts of disks as a function of $b$ is based on the Lie algebra of vector fields preserving a holomorphic 
volume form on a complex torus \cite{KSaffine}. Geometrically, the wall-crossing formula guarantees the consistency of the gluing procedure producing the mirror $M$. 

The wall-crossing formula for DT invariants as a function of the stability is formally identical to the wall-crossing formula for counts of holomorphic disks when $M$ is holomorphic symplectic: in this case, the wall-crossing transformations are symplectic and not just volume preserving as in the general Calabi-Yau situation \cite{KSint, LinK3,lu2010instanton}. 
This suggests a correspondence between DT invariants and counts of $J_\zeta$-holomorphic disks in complex integrable system \cite{KSint}. Another reason to expect this correspondence is the work of Gaiotto-Moore-Neitzke \cite{GMN1,GMN2} in which the hyperk\"ahler metric on the Seiberg-Witten integrable system $M$ is reconstructed from ``quantum corrections" determined by the BPS spectrum. As $M$ is self-mirror \cite{Derry}, compatibility between the SYZ mirror construction and the Gaiotto-Moore-Neitzke description of the hyperk\"ahler metric requires a matching between counts of BPS states and counts of $J_\zeta$-holomorphic curves. 
Conjecture \ref{conj_main} provides a categorification of this correspondence and is formulated in a way leading to a natural physics derivation presented in the next section. It would be very interesting to find more direct mathematical evidence for it. 

Mathematically rigorous examples of correspondences between DT invariants and algebro-geometric versions of counts of holomorphic disks given by punctured Gromov--Witten invariants \cite{ACGSpunctured,arguz2020higher, GSintrinsic, GScanonical} have been obtained for quiver DT invariants \cite{arguz2022quiv,MR4157555, gross2010tropical, gross2010quivers, MR3004575} and for DT invariants of local $\PP^2$ \cite{bousseau2019takahashi,bousseau2019scattering, bousseau2022bps}. In such correspondences, tropicalizations in $B$ of holomorphic disks in $M$ are interpreted as ``attractor flow trees" in the space of stability conditions for DT invariants, see Figure 
\ref{fig10}
\cite{KSint, wang}. We view these results as evidence for Conjecture \ref{conj_main}.
\end{remark}

\begin{figure}[ht!]
\centering
\includegraphics[width=40mm]{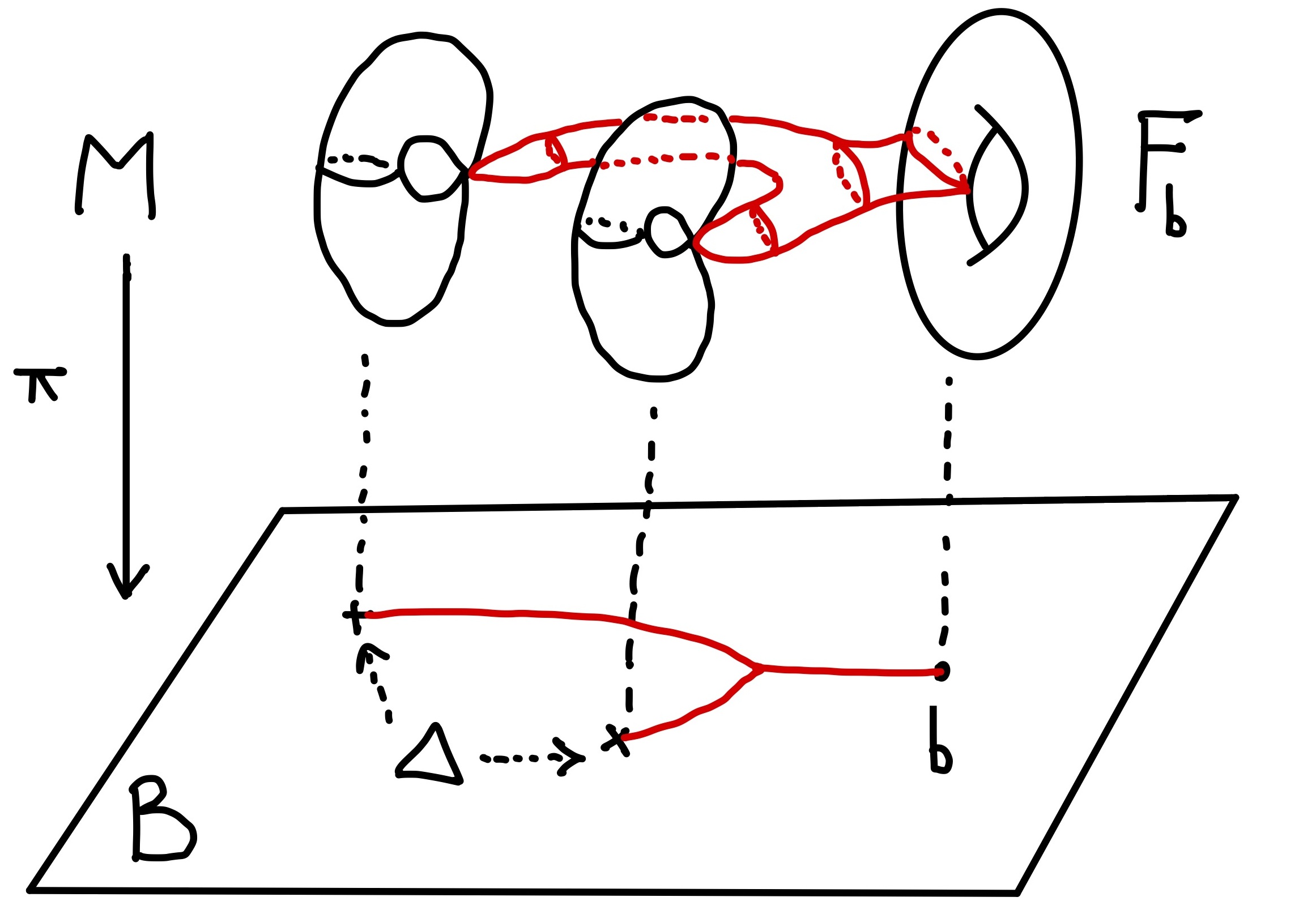}
\caption{Attractor flow trees as tropicalization of holomorphic disks with boundary on $F_b$. \label{fig10}}
\end{figure}

\begin{remark}
The complex structure $I$, the holomorphic symplectic form $\Omega_I$, and the $I$-holomorphic map $\pi \colon M \rightarrow B$ on $M$ are canonically attached to the 4d $\mathcal{N}=2$ 
theory $\mathcal{T}$. By contrast, the hyperk\"ahler metric depends on the radius $R$ of the circle $S^1$ on which one compactifies $\mathcal{T}$ to obtain a  3d $\mathcal{N}=4$ sigma model. The fibers of $\pi$ have size of order $1/R$. In the limit $R \rightarrow +\infty$, fibers collapse to the base: indeed, in this limit, the circle decompactifies, and the Coulomb branch $M$ of $\mathcal{T}$ on $\RR^3 \times S^1$ reduces to the Coulomb branch of $\mathcal{T}$ on $\RR^4$.
While one may expect that holomorphic Floer theory only depends on the $I$-holomorphic symplectic structure and not on the precise hyperk\"ahler metric, it will be probably easier to define and study it in the collapsing limit $R \rightarrow +\infty$.
\end{remark}

\begin{remark}\label{rem_2d_4d}
There is a formal 2d-4d analogy between Conjecture \ref{conj_general_2d} and Conjecture \ref{conj_main}.
In both cases, we have a space of parameters (the space $\mathcal{N}_{\mathcal{T}}$ of 2d
$\mathcal{N}=(2,2)$
theories or the Coulomb branch of vacua of the $4d$
$\mathcal{N}=2$ theory), a holomorphic symplectic manifold fibered over the space of parameters (the cotangent bundle 
$T^{*}\mathcal{N}_{\mathcal{T}} \rightarrow \mathcal{N}_{\mathcal{T}}$ or the Hitchin integrable system $\pi: M \rightarrow B$), and the space of (2d or 4d) BPS states at a given point of the parameter space is given by the holomorphic Floer theory of the holomorphic symplectic manifold for two holomorphic Lagrangian submanifolds given by the fiber of the fibration over this point and a section or multi-section (the spectral cover $\mathcal{L}$ or the Hitchin section $S$).
\end{remark}

\subsection{A physics derivation}\label{sec_derivation}
In this section, we give a heuristic derivation of Conjecture \ref{conj_main} by compactification on a cigar geometry. 
We start with the 4d $\mathcal{N}=2$
theory $\mathcal{T}$ on $\RR^2 \times D$, where $D$ is a cigar geometry reviewed at the beginning of \S\ref{sec_main_hft_conj}. As boundary condition at the end of the cigar, one imposes to the theory to approach the vacuum $b \in B \setminus \Delta$. Finally, we restrict this system to the sector of charge $\gamma \in \Gamma_b$ and we look for the space $V$ of half-supersymmetric states of lowest energy. As we assume that the system is in the vacuum $u$ at the end of the cigar, non-trivial excitations should be localized near the tip of the cigar. But near the tip of the cigar, $\RR^2 \times D$ looks like $\RR^4$. By definition, BPS states of charge $\gamma$ are the lowest energy half-supersymmetric states of the theory on $\RR^4$ in the sector of charge $\gamma$, and so one concludes that $V \simeq BPS_\gamma^b$. In this picture, the BPS particles are just sitting at the tip of the cigar. 

Now we take the limit where the circle $S^1$ in the cigar becomes very small. In this limit, the cigar $D$ reduces to the interval $[0,1]$. The low-energy effective description of $\mathcal{T}$ compactified on a circle is the 3-dimensional $\mathcal{N}=4$ sigma-model with target $M$. Therefore, our previous system reduces to a 3d $\mathcal{N}=4$ sigma model of target $M$ on $\RR^2 \times [0,1]$ with appropriate boundary conditions at the ends of $[0,1]$. 
At $r=0$, the boundary condition defined by the tip of the cigar is the section $S$, as reviewed in \ref{sec_main_hft_conj}. On the other hand, at $r=1$, the condition to be in the vacuum $b$ translates into the boundary condition imposed by the fiber $F_b$.

In the limit where the interval $[0,1]$ also becomes very small, we obtain a 2d $\mathcal{N}=(2,2)$
sigma model on $\RR^2$ with target the infinite-dimensional K\"ahler manifold $\mathcal{P}$ of paths $[0,1] \rightarrow M$ stretched between $S$ and $F_b$. 
The central charge
\[ Z_\gamma(b)=\int_\gamma \Omega_I \]
in the original 4d $\mathcal{N}=2$ supersymmetry algebra induces a holomorphic superpotential on $\cP$
which is exactly the holomorphic action functional $W$.
Therefore, we obtain as low energy effective theory on $\RR^2$ the 2d $\mathcal{N}=(2,2)$ Landau-Ginzburg model $(\mathcal{P},W)$. Moreover, the sector of charge $\gamma$ is obtained by considering only a pair of critical point of $W$ differing by $\gamma$. By definition of holomorphic Floer theory, the space $V$ of lowest energy supersymmetric states of this system is exactly the space $H((p,0),(p,\gamma))$.

In conclusion, Conjecture \ref{conj_main} follows from comparing the 4d $\mathcal{N}=2$
theory $\mathcal{T}$ on $\RR^2 \times D$ with the  2d $\mathcal{N}=(2,2)$ effective low energy on $\RR^2$ describing the theory in the limit where the cigar $D$ is small, see Figure 
\ref{fig11}

\begin{figure}[ht!]
\centering
\includegraphics[width=40mm]{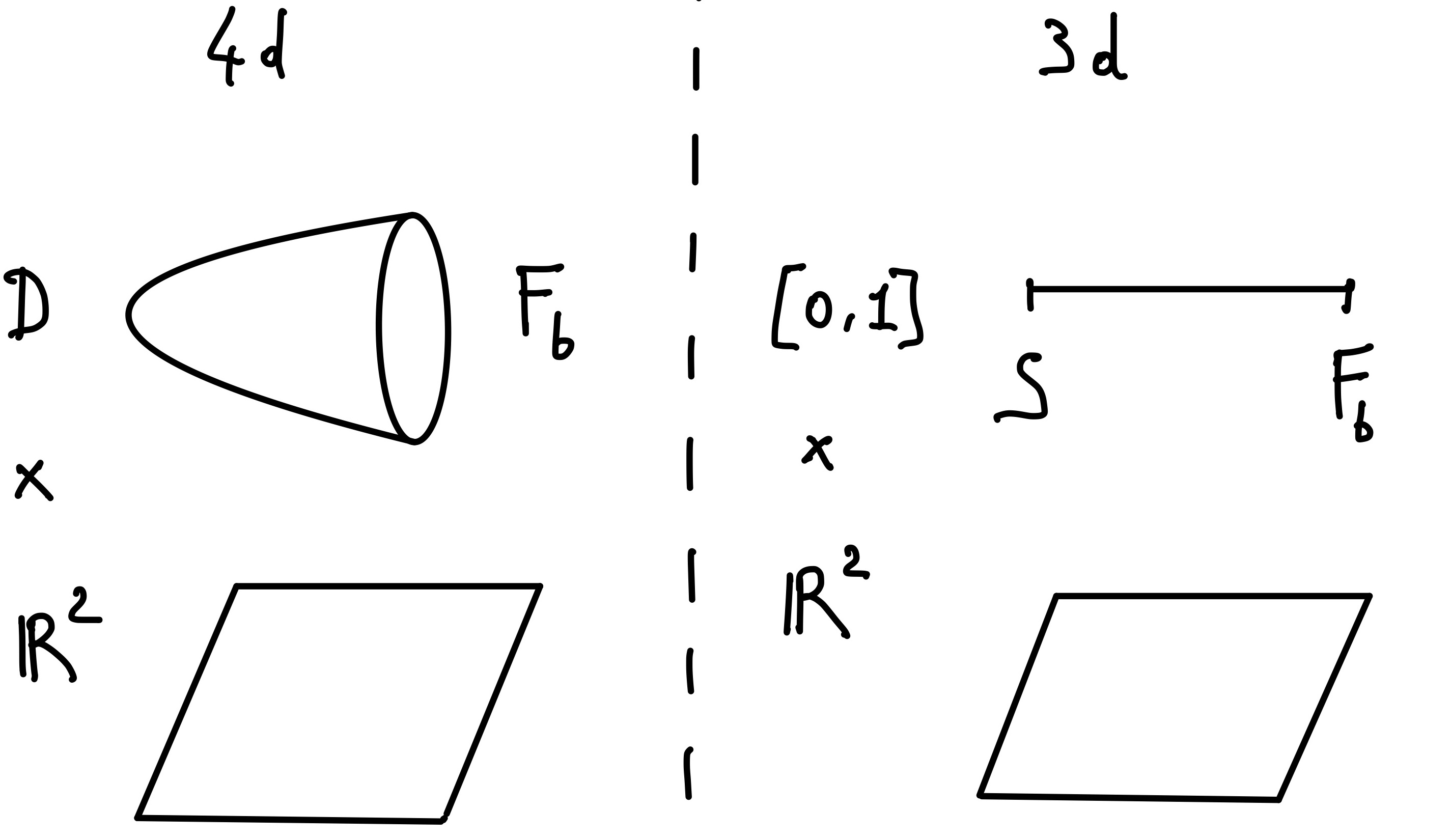}
\caption{Dimensional reduction of the cigar to an interval. \label{fig11}}
\end{figure}

\begin{remark} It is clear from this argument that the assumption that $\mathcal{T}$ is a theory without dynamical theory is necessary. 
Indeed, a 4d $\mathcal{N}=2$ supergravity theory, such as a Type II string compactification on a compact Calabi-Yau 3-fold, contains BPS black holes, and it is not possible to place a black hole at the tip of the cigar without destroying the local geometry.
\end{remark}

\begin{remark}
In the literature, compactification on a cigar geometry is often done in combination with a $\frac{1}{2}\Omega$-deformation along the $S^1$-rotations of the cigar \cite{NW}. It is essential for our purposes that the $\frac{1}{2}\Omega$-deformation parameter $\epsilon_1$ is set to $0$. For $\epsilon_1=0$, the possible supersymmetric boundary conditions at the boundary of the cigar are described by $I$-holomorphic Lagrangian submanifolds ((B,A,A)-branes) of $M$, and so we can take the fiber $F_b$ as boundary condition. By contrast, for $\epsilon_1 \neq 0$, possible supersymmetric boundary conditions are $J$-holomorphic Lagrangian submanifolds ((A,B,A)-brane)
and so the fiber $F_b$ cannot be used (correspondingly, the boundary condition corresponding to the tip of the cigar deforms the (B,A,A)-brane given by the section $S$ to a (A,B,A)-brane given by the locus of opers for class S theories \cite{NW}). On a related note,
it would be interesting to understand precisely the relation with the limit $\epsilon_1 \rightarrow 0$ of the twistorial topological string of Cecotti-Neitzke-Vafa which is defined using $\epsilon_1 \neq 0$ \cite{CNV}. Finally, Balasubramanian-Teschner consider in
\cite[\S 7.3]{BTsusy} a family of (A,B,A)-branes which in the limit $\epsilon_1 \rightarrow 0$ reduce to the (B,A,A)-branes given by the fibers $F_b$ and so it would also be interesting to see if this could be used to give a continuous deformation of our set-up (see also \cite[Eq 5.6]{HRS} and references there for further generalizations of these $(A,B,A)$ branes in terms of Fenchel-Nielsen type spectral coordinates).
\end{remark}

\subsection{Holomorphic Floer theory and categories of line operators}
\label{sec_lines}
In this section, we formulate conjectures about the categories attached by holomorphic Floer theory to the pairs of holomorphic Lagrangian submanifolds $(S,F_b)$ and $(S,S)$ in the Seiberg-Witten integrable system $M$.
It follows from the definition in \S\ref{sec_cat-pairs_lag} of the Hom categories in the Fueter 2-category that the category 
\[ C(S,F_b):=\Hom_{\mathrm{Ft}(M,\Omega_I)}(S,F_b)\,\]
admits for every $\zeta \in \CC^{*}$, $|\zeta|=1$, an exceptional collections $(X_\gamma(\zeta))_{\gamma \in \Gamma_b}$ indexed by charges $\gamma \in \Gamma_b$. In other words, $C(S,F_b)$ is a deformation of the derived category of coherent sheaves $\mathrm{D^bCoh}(BT)$ on the classifying stack of the torus $T=\Spec \CC[\Gamma_b] \simeq (\CC^{*})^n$.
The Hom spaces in $C(F,S_b)$ are constructed as in \S\ref{sec_cat-pairs_lag} from the complexes of BPS states $C((p,\gamma),(p,\gamma'))$.

Formally, $X_\gamma(\zeta)$ is the Lefschetz thimble of gradient flow lines of $\Rea(\zeta^{-1}W)$ emanating from the critical point $(p,\gamma)$
of the holomorphic action functional $W$ on the space of paths $\widetilde{\cP}_0$. The object $X_\gamma(\zeta)$ should be viewed as the $\zeta$-supersymmetric IR line defect of charge $\gamma$ in the vacuum $b$ \cite{GMNframed, cordova2014line}. Moreover, the expectation value $\langle X_\gamma(\zeta)\rangle$ of the IR line operator in the ``conformal limit" in the sense of \cite{gaiotto2014opers} should be formally given by the corresponding infinite-dimensional period integral:
\[ \langle X_\gamma(\zeta)\rangle=\int_{X_\gamma(\zeta)} e^{\Rea(\zeta^{-1}W)}vol \,,\]
where $vol$ is ``the volume form" on $\widetilde{\cP}_0$. When $\zeta$ varies, the (infinite) exceptional collection $(X_\gamma(\zeta))$ changes by mutations determined by the complexes of BPS states. In particular, the infinite-dimensional period integrals $\langle X_\gamma(\zeta)\rangle$ are solutions of the Riemann-Hilbert problems formulated in terms of BPS/DT invariants in \cite{BRH, gaiotto2014opers}.

Working with an appropriate wrapped version of the Fueter 2-category, one should be able to make sense of the category 
\[ C(S,S):=\Hom_{\mathrm{Ft}(M,\Omega_I)}(S,S)\,.\]
Similarly to what happens in the usual wrapped Fukaya category, one expects $C(S,S)$ to have for every $\zeta \in \CC$, $|\zeta|=1$,
an exceptional collection $(L_m(\zeta))_m$ indexed by the infinite set of intersection points between $S$ and a wrapped perturbation of $S$. Moreover, the composition of Hom categories in the Fueter 2-category should define a structure of monoidal category on $C(S,S)$, and a structure of module category over $C(S,S)$ on $C(S,F_b)$. Acting with $C(S,S)$ on the object $X_0(\zeta)$ of $C(S,F_b)$ obtained for $\gamma=0$ induces a functor 
\[ RG_{b}(\zeta): C(S,S) \longrightarrow C(S,F_b)\,.\]

\begin{conjecture} \label{conj_UV_lines}
The monoidal category $C(S,S)$ is the category of supersymmetric UV line operators of the 4d 
$\mathcal{N}=2$
theory $\mathcal{T}$ as in \cite{cordova2014line, GMNframed}.
\end{conjecture}

When $\mathcal{T}$ is a class S theory, $M$ is a moduli space of Higgs bundles. Let $M_B$ be the corresponding character variety: $M_B$ is an affine algebraic variety, which as a complex manifold is isomorphic to $M$ with complex structure $J$. In this case, Conjecture \ref{conj_UV_lines} predicts that $C(S,S)$ is a monoidal categorification of the algebra of regular functions $\CC[M_B]$ of $M_B$. Moreover, the objects $L_m(\zeta)$ should be categorical lifts of the canonical basis of theta functions predicted by mirror symmetry \cite{GSintrinsic, GScanonical}.

\begin{conjecture}\label{conj_RG}
The functor $RG_{b}(\zeta): C(S,S) \longrightarrow C(S,F_b)$ is a categorification of the map in \cite{cordova2014line, GMNframed} expanding UV line operators in terms of IR line operators.
\end{conjecture}

In other words, Conjecture \ref{conj_RG} predicts that 
$RG_{b}(\zeta)$ should be a categorification of the map expressing theta functions in terms of the cluster-like coordinates $\langle X_\gamma(\zeta)\rangle$.

\bibliographystyle{plain}
\bibliography{bibliography}

\begin{thebibliography}{10}

\bibitem{wrapped}
Mohammed Abouzaid and Paul Seidel.
\newblock An open string analogue of {V}iterbo functoriality.
\newblock {\em Geom. Topol.}, 14(2):627--718, 2010.

\bibitem{ACGSpunctured}
Dan Abramovich, Qile Chen, Mark Gross, and Bernd Siebert.
\newblock Punctured logarithmic maps.
\newblock {\em arXiv preprint arXiv:2009.07720}, 2020.

\bibitem{arguz2022quiv}
H{\"u}lya Arg{\"u}z and Pierrick Bousseau.
\newblock Quivers and curves in higher dimension.
\newblock {\em In preparation}, 2022.

\bibitem{arguz2020higher}
H{\"u}lya Arg{\"u}z and Mark Gross.
\newblock The higher dimensional tropical vertex.
\newblock {\em arXiv preprint arXiv:2007.08347}, 2020.

\bibitem{AF}
Michael Atiyah.
\newblock New invariants of {$3$}- and {$4$}-dimensional manifolds.
\newblock In {\em The mathematical heritage of {H}ermann {W}eyl ({D}urham,
  {NC}, 1987)}, volume~48 of {\em Proc. Sympos. Pure Math.}, pages 285--299.
  Amer. Math. Soc., Providence, RI, 1988.

\bibitem{Aud}
Mich\`ele Audin.
\newblock Symplectic geometry in {F}robenius manifolds and quantum cohomology.
\newblock {\em J. Geom. Phys.}, 25(1-2):183--204, 1998.

\bibitem{BTsusy}
Aswin Balasubramanian and J\"{o}rg Teschner.
\newblock Supersymmetric field theories and geometric {L}anglands: the other
  side of the coin.
\newblock In {\em String-{M}ath 2016}, volume~98 of {\em Proc. Sympos. Pure
  Math.}, pages 79--105. Amer. Math. Soc., Providence, RI, 2018.

\bibitem{BF_DQ}
Kai Behrend and Barbara Fantechi.
\newblock Gerstenhaber and {B}atalin-{V}ilkovisky structures on {L}agrangian
  intersections.
\newblock In {\em Algebra, arithmetic, and geometry: in honor of {Y}u. {I}.
  {M}anin. {V}ol. {I}}, volume 269 of {\em Progr. Math.}, pages 1--47.
  Birkh\"{a}user Boston, Boston, MA, 2009.

\bibitem{bousseau2019takahashi}
Pierrick Bousseau.
\newblock A proof of {N}.\ {T}akahashi's conjecture for $(\mathbb{P}^2,{E})$
  and a refined sheaves/{G}romov-{W}itten correspondence.
\newblock {\em arXiv preprint arXiv:1909.02992}, 2019.

\bibitem{MR4157555}
Pierrick Bousseau.
\newblock The quantum tropical vertex.
\newblock {\em Geom. Topol.}, 24(3):1297--1379, 2020.

\bibitem{bousseau2019scattering}
Pierrick Bousseau.
\newblock Scattering diagrams, stability conditions, and coherent sheaves on
  {$\mathbb{P}^2$}.
\newblock {\em J. Algebraic Geom.}, 31(4):593--686, 2022.

\bibitem{bousseau2022bps}
Pierrick Bousseau, Pierre Descombes, Bruno~Le Floch, and Boris Pioline.
\newblock {BPS} {D}endroscopy on {L}ocal $\mathbb{P}^2$.
\newblock {\em arXiv preprint arXiv:2210.10712}, 2022.

\bibitem{J_DQ}
Christopher Brav, Vittoria Bussi, Delphine Dupont, Dominic Joyce, and
  Bal{\'a}zs Szendroi.
\newblock Symmetries and stabilization for sheaves of vanishing cycles.
\newblock {\em arXiv preprint arXiv:1211.3259}, 2012.

\bibitem{MR2373143}
Tom Bridgeland.
\newblock Stability conditions on triangulated categories.
\newblock {\em Ann. of Math. (2)}, 166(2):317--345, 2007.

\bibitem{BRH}
Tom Bridgeland.
\newblock Riemann-{H}ilbert problems from {D}onaldson-{T}homas theory.
\newblock {\em Invent. Math.}, 216(1):69--124, 2019.

\bibitem{BgeomDT}
Tom Bridgeland.
\newblock Geometry from {D}onaldson-{T}homas invariants.
\newblock In {\em Integrability, quantization, and geometry {II}. {Q}uantum
  theories and algebraic geometry}, volume 103 of {\em Proc. Sympos. Pure
  Math.}, pages 1--66. Amer. Math. Soc., Providence, RI, [2021] \copyright
  2021.

\bibitem{BS}
Tom Bridgeland and Ivan Smith.
\newblock Quadratic differentials as stability conditions.
\newblock {\em Publ. Math. Inst. Hautes \'{E}tudes Sci.}, 121:155--278, 2015.

\bibitem{BTL}
Tom Bridgeland and Valerio Toledano~Laredo.
\newblock Stability conditions and {S}tokes factors.
\newblock {\em Invent. Math.}, 187(1):61--98, 2012.

\bibitem{CNV}
Sergio Cecotti, Andrew Neitzke, and Cumrun Vafa.
\newblock Twistorial topological strings and a {$tt^*$} geometry for {$N=2$}
  theories in {$4d$}.
\newblock {\em Adv. Theor. Math. Phys.}, 20(2):193--312, 2016.

\bibitem{CV}
Sergio Cecotti and Cumrun Vafa.
\newblock On classification of {$N=2$} supersymmetric theories.
\newblock {\em Comm. Math. Phys.}, 158(3):569--644, 1993.

\bibitem{cordova2014line}
Clay C\'{o}rdova and Andrew Neitzke.
\newblock Line defects, tropicalization, and multi-centered quiver quantum
  mechanics.
\newblock {\em J. High Energy Phys.}, (9):099, front matter+65, 2014.

\bibitem{Derry}
Richard Derryberry.
\newblock Stacky dualities for the moduli of {H}iggs bundles.
\newblock {\em Adv. Math.}, 368:107152, 55, 2020.

\bibitem{DR}
Aleksander Doan and Semon Rezchikov.
\newblock Holomorphic {F}loer theory and the {F}ueter equation.
\newblock {\em arXiv:2210.12047}, 2022.

\bibitem{DS}
Simon Donaldson and Ed~Segal.
\newblock Gauge theory in higher dimensions, {II}.
\newblock In {\em Surveys in differential geometry. {V}olume {XVI}. {G}eometry
  of special holonomy and related topics}, volume~16 of {\em Surv. Differ.
  Geom.}, pages 1--41. Int. Press, Somerville, MA, 2011.

\bibitem{DTgauge}
Simon Donaldson and Richard Thomas.
\newblock Gauge theory in higher dimensions.
\newblock In {\em The geometric universe ({O}xford, 1996)}, pages 31--47.
  Oxford Univ. Press, Oxford, 1998.

\bibitem{Dub}
Boris Dubrovin.
\newblock Geometry of {$2$}{D} topological field theories.
\newblock In {\em Integrable systems and quantum groups ({M}ontecatini {T}erme,
  1993)}, volume 1620 of {\em Lecture Notes in Math.}, pages 120--348.
  Springer, Berlin, 1996.

\bibitem{FJR}
Huijun Fan, Tyler Jarvis, and Yongbin Ruan.
\newblock The {W}itten equation, mirror symmetry, and quantum singularity
  theory.
\newblock {\em Ann. of Math. (2)}, 178(1):1--106, 2013.

\bibitem{Flo1}
Andreas Floer.
\newblock Morse theory for {L}agrangian intersections.
\newblock {\em J. Differential Geom.}, 28(3):513--547, 1988.

\bibitem{Flo2}
Andreas Floer.
\newblock Witten's complex and infinite-dimensional {M}orse theory.
\newblock {\em J. Differential Geom.}, 30(1):207--221, 1989.

\bibitem{Fukcat}
Kenji Fukaya.
\newblock Morse homotopy, {$A_\infty$}-category, and {F}loer homologies.
\newblock In {\em Proceedings of {GARC} {W}orkshop on {G}eometry and {T}opology
  '93 ({S}eoul, 1993)}, volume~18 of {\em Lecture Notes Ser.}, pages 1--102.
  Seoul Nat. Univ., Seoul, 1993.

\bibitem{Fuk}
Kenji Fukaya.
\newblock Multivalued {M}orse theory, asymptotic analysis and mirror symmetry.
\newblock In {\em Graphs and patterns in mathematics and theoretical physics},
  volume~73 of {\em Proc. Sympos. Pure Math.}, pages 205--278. Amer. Math.
  Soc., Providence, RI, 2005.

\bibitem{MR1480992}
Kenji Fukaya and Yong-Geun Oh.
\newblock Zero-loop open strings in the cotangent bundle and {M}orse homotopy.
\newblock {\em Asian J. Math.}, 1(1):96--180, 1997.

\bibitem{gaiotto2014opers}
Davide Gaiotto.
\newblock Opers and {TBA}.
\newblock {\em arXiv preprint arXiv:1403.6137}, 2014.

\bibitem{GMN1}
Davide Gaiotto, Gregory Moore, and Andrew Neitzke.
\newblock Four-dimensional wall-crossing via three-dimensional field theory.
\newblock {\em Comm. Math. Phys.}, 299(1):163--224, 2010.

\bibitem{GMN_2d_4d}
Davide Gaiotto, Gregory Moore, and Andrew Neitzke.
\newblock Wall-crossing in coupled 2d-4d systems.
\newblock {\em J. High Energy Phys.}, (12):082, front matter + 166, 2012.

\bibitem{GMNframed}
Davide Gaiotto, Gregory Moore, and Andrew Neitzke.
\newblock Framed {BPS} states.
\newblock {\em Adv. Theor. Math. Phys.}, 17(2):241--397, 2013.

\bibitem{GMN2}
Davide Gaiotto, Gregory Moore, and Andrew Neitzke.
\newblock Wall-crossing, {H}itchin systems, and the {WKB} approximation.
\newblock {\em Adv. Math.}, 234:239--403, 2013.

\bibitem{GMWinfrared}
Davide Gaiotto, Gregory Moore, and Edward Witten.
\newblock Algebra of the {I}nfrared: {S}tring {F}ield {T}heoretic {S}tructures
  in {M}assive ${N}=(2, 2) $ {F}ield {T}heory {I}n {T}wo {D}imensions.
\newblock {\em arXiv preprint arXiv:1506.04087}, 2015.

\bibitem{gammage2022perverse2}
Benjamin Gammage and Justin Hilburn.
\newblock Betti {T}ate's thesis and the trace of perverse schobers.
\newblock {\em arXiv preprint arXiv:2210.06548}, 2022.

\bibitem{gammage2022perverse}
Benjamin Gammage, Justin Hilburn, and Aaron Mazel-Gee.
\newblock Perverse schobers and 3d mirror symmetry.
\newblock {\em arXiv preprint arXiv:2202.06833}, 2022.

\bibitem{gross2010quivers}
Mark Gross and Rahul Pandharipande.
\newblock Quivers, curves, and the tropical vertex.
\newblock {\em Portugaliae Mathematica}, 67(2):211--259, 2010.

\bibitem{gross2010tropical}
Mark Gross, Rahul Pandharipande, and Bernd Siebert.
\newblock The tropical vertex.
\newblock {\em Duke Mathematical Journal}, 153(2):297--362, 2010.

\bibitem{GSaffine}
Mark Gross and Bernd Siebert.
\newblock From real affine geometry to complex geometry.
\newblock {\em Ann. of Math. (2)}, 174(3):1301--1428, 2011.

\bibitem{GSintrinsic}
Mark Gross and Bernd Siebert.
\newblock Intrinsic mirror symmetry.
\newblock {\em arXiv preprint arXiv:1909.07649}, 2019.

\bibitem{GScanonical}
Mark Gross and Bernd Siebert.
\newblock The canonical wall structure and intrinsic mirror symmetry.
\newblock {\em Invent. Math.}, 229(3):1101--1202, 2022.

\bibitem{Hayd}
Andriy Haydys.
\newblock Fukaya-{S}eidel category and gauge theory.
\newblock {\em J. Symplectic Geom.}, 13(1):151--207, 2015.

\bibitem{MR2555940}
Sonja Hohloch, Gregor Noetzel, and Dietmar~A. Salamon.
\newblock Floer homology groups in hyperk\"ahler geometry.
\newblock In {\em New perspectives and challenges in symplectic field theory},
  volume~49 of {\em CRM Proc. Lecture Notes}, pages 251--261. Amer. Math. Soc.,
  Providence, RI, 2009.

\bibitem{HRS}
Lotte Hollands, Philipp R\"{u}ter, and Richard~J. Szabo.
\newblock A geometric recipe for twisted superpotentials.
\newblock {\em J. High Energy Phys.}, (12):Paper No. 164, 90, 2021.

\bibitem{K3_coulomb}
Kenneth Intriligator.
\newblock Compactified little string theories and compact moduli spaces of
  vacua.
\newblock {\em Phys. Rev. D (3)}, 61(10):106005, 6, 2000.

\bibitem{Jhol}
Dominic Joyce.
\newblock Holomorphic generating functions for invariants counting coherent
  sheaves on {C}alabi-{Y}au 3-folds.
\newblock {\em Geom. Topol.}, 11:667--725, 2007.

\bibitem{JS}
Dominic Joyce and Yinan Song.
\newblock A theory of generalized {D}onaldson-{T}homas invariants.
\newblock {\em Mem. Amer. Math. Soc.}, 217(1020):iv+199, 2012.

\bibitem{KKS}
Mikhail Kapranov, Maxim Kontsevich, and Yan Soibelman.
\newblock Algebra of the infrared and secondary polytopes.
\newblock {\em Adv. Math.}, 300:616--671, 2016.

\bibitem{kapranov2014perverse}
Mikhail Kapranov and Vadim Schechtman.
\newblock Perverse {S}chobers.
\newblock {\em arXiv preprint arXiv:1411.2772}, 2014.

\bibitem{kapranov2020perverse}
Mikhail Kapranov, Yan Soibelman, and Lev Soukhanov.
\newblock Perverse schobers and the {A}lgebra of the {I}nfrared.
\newblock {\em arXiv preprint arXiv:2011.00845}, 2020.

\bibitem{kapustin2005branes}
Anton Kapustin.
\newblock A-branes and noncommutative geometry.
\newblock {\em arXiv preprint hep-th/0502212}, 2005.

\bibitem{MR2771578}
Anton Kapustin and Lev Rozansky.
\newblock Three-dimensional topological field theory and symplectic algebraic
  geometry {II}.
\newblock {\em Commun. Number Theory Phys.}, 4(3):463--549, 2010.

\bibitem{MR2522724}
Anton Kapustin, Lev Rozansky, and Natalia Saulina.
\newblock Three-dimensional topological field theory and symplectic algebraic
  geometry. {I}.
\newblock {\em Nuclear Phys. B}, 816(3):295--355, 2009.

\bibitem{KS_DQ}
Masaki Kashiwara and Pierre Schapira.
\newblock Deformation quantization modules.
\newblock {\em Ast\'{e}risque}, (345):xii+147, 2012.

\bibitem{KKP}
Ludmil Katzarkov, Maxim Kontsevich, and Tony Pantev.
\newblock Hodge theoretic aspects of mirror symmetry.
\newblock In {\em From {H}odge theory to integrability and {TQFT}
  tt*-geometry}, volume~78 of {\em Proc. Sympos. Pure Math.}, pages 87--174.
  Amer. Math. Soc., Providence, RI, 2008.

\bibitem{khan2021categorical}
Ahsan Khan.
\newblock {\em Categorical Aspects of BPS States}.
\newblock PhD thesis, Rutgers The State University of New Jersey, School of
  Graduate Studies, 2021.

\bibitem{khan2020categorical}
Ahsan Khan and Gregory Moore.
\newblock Categorical {W}all-{C}rossing in {L}andau-{G}inzburg {M}odels.
\newblock {\em arXiv preprint arXiv:2010.11837}, 2020.

\bibitem{MR1855264}
Maxim Kontsevich.
\newblock Deformation quantization of algebraic varieties.
\newblock {\em Lett. Math. Phys.}, 56(3):271--294, 2001.
\newblock EuroConf\'erence Mosh\'e Flato 2000, Part III (Dijon).

\bibitem{KSaffine}
Maxim Kontsevich and Yan Soibelman.
\newblock Affine structures and non-{A}rchimedean analytic spaces.
\newblock In {\em The unity of mathematics}, volume 244 of {\em Progr. Math.},
  pages 321--385. Birkh\"auser Boston, Boston, MA, 2006.

\bibitem{kontsevich2008stability}
Maxim Kontsevich and Yan Soibelman.
\newblock Stability structures, motivic {D}onaldson-{T}homas invariants and
  cluster transformations.
\newblock {\em arXiv preprint arXiv:0811.2435}, 2008.

\bibitem{KSint}
Maxim Kontsevich and Yan Soibelman.
\newblock Wall-crossing structures in {D}onaldson-{T}homas invariants,
  integrable systems and mirror symmetry.
\newblock In {\em Homological mirror symmetry and tropical geometry}, volume~15
  of {\em Lect. Notes Unione Mat. Ital.}, pages 197--308. Springer, Cham, 2014.

\bibitem{KS_HFT}
Maxim Kontsevich and Yan Soibelman.
\newblock Holomorphic {F}loer {T}heory: brane quantization, exponential
  integrals and resurgence.
\newblock {\em In preparation}, 2022.

\bibitem{LinK3}
Yu-Shen Lin.
\newblock Correspondence theorem between holomorphic discs and tropical discs
  on {K}3 surfaces.
\newblock {\em J. Differential Geom.}, 117(1):41--92, 2021.

\bibitem{lu2010instanton}
Wenxuan Lu.
\newblock {I}nstanton {C}orrection, {W}all {C}rossing {A}nd {M}irror {S}ymmetry
  {O}f {H}itchin's {M}oduli {S}paces.
\newblock {\em arXiv preprint arXiv:1010.3388}, 2010.

\bibitem{lurie}
Jacob Lurie.
\newblock On the classification of topological field theories.
\newblock In {\em Current developments in mathematics, 2008}, pages 129--280.
  Int. Press, Somerville, MA, 2009.

\bibitem{M_DQ}
Borislav Mladenov.
\newblock Formality of differential graded algebras and complex lagrangian
  submanifolds.
\newblock {\em arXiv preprint arXiv:2007.05498}, 2020.

\bibitem{NW}
Nikita Nekrasov and Edward Witten.
\newblock The {O}mega deformation, branes, tntegrability and {L}iouville
  theory.
\newblock {\em J. High Energy Phys.}, (9):092, i, 82, 2010.

\bibitem{P_DQ}
Jonathan Pridham.
\newblock Quantisation of derived lagrangians.
\newblock {\em arXiv preprint arXiv:1607.01000}, 2016.

\bibitem{MR3004575}
Markus Reineke and Thorsten Weist.
\newblock Refined {GW}/{K}ronecker correspondence.
\newblock {\em Math. Ann.}, 355(1):17--56, 2013.

\bibitem{RW}
Lev Rozansky and Edward Witten.
\newblock Hyper-{K}\"{a}hler geometry and invariants of three-manifolds.
\newblock {\em Selecta Math. (N.S.)}, 3(3):401--458, 1997.

\bibitem{Hodge_nc}
G\"{u}nter Schwarz.
\newblock {\em Hodge decomposition---a method for solving boundary value
  problems}, volume 1607 of {\em Lecture Notes in Mathematics}.
\newblock Springer-Verlag, Berlin, 1995.

\bibitem{MR1293681}
Nathan Seiberg and Edward Witten.
\newblock Electric-magnetic duality, monopole condensation, and confinement in
  {$N=2$} supersymmetric {Y}ang-{M}ills theory.
\newblock {\em Nuclear Phys. B}, 426(1):19--52, 1994.

\bibitem{MR1306869}
Nathan Seiberg and Edward Witten.
\newblock Monopoles, duality and chiral symmetry breaking in {$N=2$}
  supersymmetric {QCD}.
\newblock {\em Nuclear Phys. B}, 431(3):484--550, 1994.

\bibitem{SW_3d}
Nathan Seiberg and Edward Witten.
\newblock Gauge dynamics and compactification to three dimensions.
\newblock In {\em The mathematical beauty of physics ({S}aclay, 1996)},
  volume~24 of {\em Adv. Ser. Math. Phys.}, pages 333--366. World Sci. Publ.,
  River Edge, NJ, 1997.

\bibitem{Seid1}
Paul Seidel.
\newblock Vanishing cycles and mutation.
\newblock In {\em European {C}ongress of {M}athematics, {V}ol. {II}
  ({B}arcelona, 2000)}, volume 202 of {\em Progr. Math.}, pages 65--85.
  Birkh\"{a}user, Basel, 2001.

\bibitem{Seid2}
Paul Seidel.
\newblock {\em Fukaya categories and {P}icard-{L}efschetz theory}.
\newblock Zurich Lectures in Advanced Mathematics. European Mathematical
  Society (EMS), Z\"{u}rich, 2008.

\bibitem{Smith1}
Ivan Smith.
\newblock Quiver algebras as {F}ukaya categories.
\newblock {\em Geom. Topol.}, 19(5):2557--2617, 2015.

\bibitem{SV}
Jake Solomon and Misha Verbitsky.
\newblock Locality in the {F}ukaya category of a hyperk\"{a}hler manifold.
\newblock {\em Compos. Math.}, 155(10):1924--1958, 2019.

\bibitem{syz}
Andrew Strominger, Shing-Tung Yau, and Eric Zaslow.
\newblock Mirror symmetry is {$T$}-duality.
\newblock {\em Nuclear Phys. B}, 479(1-2):243--259, 1996.

\bibitem{FS_rig}
Donghao Wang.
\newblock The {C}omplex {G}radient {F}low equation and {S}eidel's {S}pectral
  {S}equence.
\newblock {\em arXiv preprint arXiv:2209.02810}, 2022.

\bibitem{wang}
Qiang Wang.
\newblock Wall-crossing structures and application to {${\rm SU}(3)$}
  {S}eiberg-{W}itten integrable systems.
\newblock {\em J. Geom. Phys.}, 157:103834, 15, 2020.

\bibitem{WMorse}
Edward Witten.
\newblock Supersymmetry and {M}orse theory.
\newblock {\em J. Differential Geometry}, 17(4):661--692 (1983), 1982.

\bibitem{Wanalytic}
Edward Witten.
\newblock Analytic continuation of {C}hern-{S}imons theory.
\newblock In {\em Chern-{S}imons gauge theory: 20 years after}, volume~50 of
  {\em AMS/IP Stud. Adv. Math.}, pages 347--446. Amer. Math. Soc., Providence,
  RI, 2011.

\end{thebibliography}

\end{document}